\documentclass[10pt]{article}
\usepackage{fullpage,epsfig,graphics,cancel,bbm,slashed}
\usepackage{amsthm,amssymb,amsmath,amsfonts,amscd,amsbsy}%,upgreek}
\usepackage{lmodern}% http://ctan.org/pkg/lmodern
\usepackage{slantsc}% http://ctan.org/pkg/slantsc
\usepackage{psfrag,hyperref,color}
\usepackage[]{youngtab}
\usepackage{graphicx}
\usepackage{wasysym}
\usepackage{mathdots}
\usepackage{mathrsfs}

\usepackage{tikz}
\usetikzlibrary{arrows,chains,matrix,positioning,scopes,snakes,decorations.pathmorphing,bending}
\usepackage{tikz-cd}
\tikzset{snake it/.style={decorate, decoration=snake}}

\makeatletter
\tikzset{join/.code=\tikzset{after node path={%
\ifx\tikzchainprevious\pgfutil@empty\else(\tikzchainprevious)%
edge[every join]#1(\tikzchaincurrent)\fi}}}
\makeatother

\tikzset{>=stealth',every on chain/.append style={join},
         every join/.style={->}}
\tikzset{
    >=stealth',
    punkt/.style={
           rectangle,
           rounded corners,
           draw=black, very thick,
           text width=6.5em,
           minimum height=2em,
           text centered},
    pil/.style={
           ->,
           thick,
           shorten <=2pt,
           shorten >=2pt,}
}

\newcommand{\BB}{\mathbb}

\newcommand{\FR}{\mathfrak}

\newcommand{\bea}{\begin{eqnarray}}
\newcommand{\eea}{\end{eqnarray}}
\newcommand{\nn}{\nonumber}
\newcommand{\Tr}{\operatorname{Tr}}

\newcommand{\tensor}{\vcenter{\hbox{\tiny$\bigotimes$}}}

\newcommand{\bra}{\langle}
\newcommand{\ket}{\rangle}

\newcommand{\re}{\operatorname{Re}}
\newcommand{\To}{\Rightarrow}

\newcommand{\opn}{\operatorname}

\newcommand\hcancel[2][black]{\setbox0=\hbox{$#2$}%
\rlap{\raisebox{.45\ht0}{\textcolor{#1}{\rule{\wd0}{1pt}}}}#2}
\def\ga{\alpha}
\def\gb{\beta}
\def\gc{\gamma}
\def\Gc{\Gamma}
\def\gd{\delta}
\def\Gd{\Delta}
\def\ep{\epsilon}

\def\gs{\sigma}

\def\gl{\lambda}

\def\Go{\Omega}

\DeclareMathAlphabet{\mathpzc}{OT1}{pzc}{m}{it}

\newtheorem{theorem}{Theorem}[section]
\newtheorem{lemma}[theorem]{Lemma}
\newtheorem{proposition}[theorem]{Proposition}

\theoremstyle{definition}
\newtheorem{example}[theorem]{Example}
\newtheorem{remark}[theorem]{Remark}
\newtheorem{notation}[theorem]{Notation}
\newtheorem{corollary}[theorem]{Corollary}
\newtheorem{definition}[theorem]{Definition}

\numberwithin{equation}{section}
\begin{document}
\begin{flushright} \small
UUITP-01/24
 \end{flushright}
\smallskip
\begin{center} \Large
{\bf Pl\"ucker Coordinates and the Rosenfeld Planes}
 \\[12mm] \normalsize
{\bf Jian Qiu${}^{a,b}$} \\[8mm]
 {\small\it
    ${}^a$Department of Physics and Astronomy, Uppsala University,\\
        \vspace{.3cm}
      ${}^b$ Mathematics Institute,  Uppsala University, \\
      \vspace{.3cm}
   Uppsala, Sweden\\  }
\end{center}
\vspace{7mm}
\begin{abstract}
 \noindent
The exceptional compact hermitian symmetric space EIII is the quotient $E_6/Spin(10)\times_{\mathbb{Z}_4}U(1)$.
We introduce the Plücker coordinates which give an embedding of EIII into $\mathbb{C}P^{26}$ as a projective subvariety.
The subvariety is cut out by 27 Plücker relations.
We show that, using Clifford algebra, one can solve this over-determined system of relations, giving local coordinate charts to the space.

Our motivation is to understand EIII as the complex projective octonion plane $(\mathbb{C}\otimes\mathbb{O})P^2$, whose construction is somewhat scattered across the literature.
We will see that the EIII has an atlas whose transition functions have clear octonion interpretations, apart from those covering a sub-variety $X_{\infty}$ of dimension 10. This subvariety is itself a hermitian symmetric space known as DIII, with no apparent octonion interpretation. We give detailed analysis of the geometry in the neighbourhood of $X_{\infty}$.

We further decompose $X={\rm EIII}$ into $F_4$-orbits: $X=Y_0\cup Y_{\infty}$, where $Y_0\sim(\mathbb{O}P^2)_{\mathbb{C}}$ is an open $F_4$-orbit and is the complexification of $\mathbb{O}P^2$, whereas $Y_{\infty}$ has co-dimension 1, thus EIII could be more appropriately denoted as $\overline{(\mathbb{O}P^2)_{\mathbb{C}}}$. This decomposition appears in the classification of equivariant completion of homogeneous algebraic varieties by Ahiezer \cite{Ahiezer}.

\end{abstract}

\eject
\normalsize
\tableofcontents
\section{Introduction}
According to the classification of compact hermitian symmetric spaces (chapter 8 \cite{wolf2011spaces}), there are four infinite families denoted AIII: $SU(p+q)/S(U(p)\times U(q))$, DIII: $SO(2n)/U(n)$, CI: $Sp(n)/U(n)$ and BDI${}_2$: $SO(p+q)/SO(p)\times SO(q)$ for $p$ or $q$ equal to 2. There are also two exceptional ones denoted EIII and EVII that are quotients of $E_6$ and $E_7$ respectively
\bea {\rm EIII}:~~E_6/Spin(10)\times_{\BB{Z}_4}U(1);~~~~ {\rm EVII}:~~E_7/E_6\times U(1)\label{Rosenfeld}.\eea
What is interesting about these spaces is that they all possess a bi-hamiltonian integrable structure.
In \cite{Bonechi:2015iha}, we studied the integrability of the four classical cases A, D, C and BD and gave a uniform treatment of the involutive hamiltonians. However, the exceptional case turned out quite recalcitrant, only a partial solution was given in \cite{Bonechi:2021ezp}. In a coming paper \cite{Depot}, we solve the EIII case.

The current paper serves two purposes. Firstly, in solving EIII, we made crucial use of a very useful coordinate system which we call the \emph{Pl\"ucker coordinates} due to their similarity with the namesake coordinates on the Grassmannians. That such coordinates should exist follows directly from the fact that the space is a projective variety. But to use these coordinates effectively for the calculation, we need to made them very concrete. This is what is done in this paper, in a way, this paper serves as a tech depot for \cite{Depot}.

Secondly, in the process of the calculation, we become interested in the relationship of EIII and EVII to the real octonion projective planes
\bea \BB{O}\BB{P}^2=F_4/Spin(9).\label{Moufang}\eea
More specifically, EIII is expected to be $(\BB{C}\otimes\BB{O})P^2$ while EVII is related to $(\BB{H}\otimes\BB{O})P^2$.

The projective space \eqref{Moufang} has a construction using exceptional Jordan algebra of $3\times3$ octonion matrices.
But similar constructions for $(\BB{C}\otimes\BB{O})P^2$ or $(\BB{H}\otimes\BB{O})P^2$ remain folklore and are dispersed in the literature.
One can complexify/modify the previous octonion matrix in such a way that the entries are precisely the Pl\"ucker coordinates, while the Pl\"ucker relations express certain co-linearity. We refer the reader to sec 4.4 of \cite{Baez:2001dm} for a review on the current status of the subject.

We will review the fact that for $X$=EIII, there are 27 Pl\"ucker coordinates that induce an embedding of $X$ into $\BB{C}P^{26}$.
The image of the embedding is cut out by 27 algebraic equations: the Pl\"ucker relations.
We will proceed to use these coordinates to analyse the geometry of the space. In particular, we give coordinate charts to $X$ and the transition functions in between. Our hope was that the transition functions would allow for a clear octonion interpretation, and thereby justifying $X$ being called the complex octonion projective plane. But we found, quite interestingly, that there is a locus $X_{\infty}\subset X$ of complex dimension 10, which is itself a hermitian symmetric space
\bea X\supset X_{\infty}=DIII=SO(10)/U(5),\nn\eea
i.e. it is the Grassmannian of Lagrangian subspace in $\BB{R}^{10}$.
The coordinate chart in the neighbourhood of $X_{\infty}$ does not seem to permit a natural interpretation as octonions.

Though the exceptional groups have elegant octonion constructions, however, octonions remain unfamiliar to most people. Even for experts, they can be unwieldy for technical computations. In this paper, we rather approach the exceptional groups using Adam's construction, based on spinors and Clifford algebra. That Clifford algebra should enter is not surprising:
the triality of the representation of $Spin(8)$ is what underlies the octonion algebra. In our approach the representations of $Spin(10)$ will be crucial in terms of organising the computation and making the Pl\"ucker coordinates/relations concrete.

To make everything easily accessible for the reader, we have a large appendix, reviewing the essentials of the Clifford algebra, Spin groups and how the exceptional groups arise from these.

The rest of the paper is organised as follows:
We will first review the familiar Pl\"ucker coordinates of Grassmannians and then move to the exceptional case. Here the 27 dimensional representation of $E_6$ plays a star role. In fact, 27 is the dimension of a $3\times3$ hermitian octonion matrix, while the Jordan algebra structure arises from a cubic invariant tensor of the 27 dimensional representation.
We then obtain the coordinate patches for $X$ and the transition function between the patches. We further study the $F_4$-decomposition of $X$, and show that $X$ is the equivariant completion of the complexification of the octonion projective plane.
The punch line of the article is that EIII is the completion of the naive complexification of $\BB{O}\BB{P}^2=F_4/Spin(9)$ in the sense of \cite{Ahiezer} (see prop.\ref{prop_punch_line}). Therefore
\bea {\rm EIII}=E_6/Spin(10)\times_{\BB{Z}_4}U(1)=\overline{(\mathbb{O}P^2)_{\mathbb{C}}}.\label{punch_line}\eea

\vskip .5cm

\noindent {\bf Acknowledgements} Firstly, I would like to thank Seidon Alsaody, who is the resident expert in octonions at the institute and I benefitted from many hours of discussion with him. Secondly, the anonymous referee also brought invaluable suggestions to the paper, both in terms of mathematical precision and possible applications.
Lastly, I need to mention that I was attracted to this subject after attending a course in group theory given by professor Pierre Ramond. I have been thinking about the octonion projective planes and exceptional groups on and off (mostly off though) in the ensuing twenty years or so. Here I would like to express my somewhat belated gratitude to prof. Ramond.

\section{Pl\"ucker Coordinates}
Before going into the Pl\"ucker coordinates for EIII, let us review the more familiar case for the Grassmannians.
\subsection{The $Gr(2,4)$ case}
The Grassmannian $Gr(2,4)$ is the space of two planes in $V=\BB{C}^4$.
We can get to the Pl\"ucker coordinates most concretely as follows: take a matrix $\{g_{ij}\}\in Gl(V)$, take its first two columns and let
\bea z_{ij}=\det
\begin{bmatrix}
  g_{i1} & g_{i2} \\ g_{j1} & g_{j2}
\end{bmatrix}.\label{Plucker_naive}\eea
The Pl\"ucker relation
\bea z_{12}z_{34}+z_{13}z_{42}+z_{14}z_{23}=0\label{Plucker_rel}\eea
follows from an explicit check.

We now recast the Pl\"ucker coordinates in the framework of representations of $Gl(V)$.
We introduce a line bundle, of which the $\{z_{ij}\}$'s are the holomorphic sections. Consider the irreducible representation $R=\wedge^2V$ of $Gl(V)$.
Let $v_0$ be the highest weight vector of $R$. Let $L$ be the complex line $\BB{C}v_0$, name the subgroup that preserves this line as $P$. It is a subgroup that clearly contains the Borel subgroup, and so is called \emph{standard parabolic subgroup}.
We have the associated line bundle ${\cal L}$ \bea
  \begin{tikzpicture}
  \matrix (m) [matrix of math nodes, row sep=2em, column sep=2em]
    {  Gl(V)\times_PL & L \\
       Gl(V)/P &  \\ };
  \path[->,font=\scriptsize]
  (m-1-1) edge node[right] {$\pi$} (m-2-1)
  (m-1-2) edge (m-1-1);
\end{tikzpicture}\nn\eea
The $z_{ij}$'s are the holomorphic sections of the \emph{dual} bundle ${\cal L}^{\vee}$. Recall that for an associated bundle, the sections are in 1-1 correspondence with the $P$-equivariant maps
\bea \gs:\,Gl(V)\to \BB{C},~~~\gs(gp)=p^{-1}\circ \gs(g)\nn\eea
where $g\in Gl(V)$, $p\in P$ and $p\circ $ means the action of $p$ on the fibre $\BB{C}$.
Indeed, given $\gs$, we define a section $s$ as $s([g])=[g,\gs(g)]$, where $[g]\in Gl(V)/P$ and $g$ is any representative of $[g]$. The choice does not matter since
$[gp,\gs(gp)]=[gp,p^{-1}\circ\gs(g)]=[g,\gs(g)]$.

To write down such equivariant maps, we fix a standard basis $\{e_i\}$ for $V=\BB{C}^4$ and the dual basis $\{e_i^{\vee}\}$ for $V^*$. We also fix the standard Borel subgroup, then the highest weight for $R$ is $v_0=e_1\wedge e_2$.
The subgroup preserving $v_0$ \emph{up to a scalar} has a block form
\bea P=
\begin{bmatrix}
  p_{11} & p_{12} \\ 0 & p_{22}
\end{bmatrix}\nn\eea
where $p_{ij}$ are $2\times2$ matrices and $p_{11}$ acts on $v_0$ as $\det p_{11}$ while $p_{12},\,p_{22}$ act trivially.
The orbit of $v_0$, upon projectivising, is precisely $Gl(V)/P$.

Now we can write down the equivariant maps $z_{ij}:\,Gl(V)\to\BB{C}$
\bea z_{ij}(g)=\bra e^{\vee}_i\wedge e^{\vee}_j|g|e_1\wedge e_2\ket,\nn\eea
where we used the physics notation of bra-ket: $|e_1\wedge e_2\ket$ is a vector in $\wedge^2V$, and $\bra e^{\vee}_i\wedge e^{\vee}_j|$ is a vector in $\wedge^2V^*$, and $\bra\cdots|\cdots\ket$ is the natural pairing between $\wedge^2V^*,\wedge^2V$.
If we unwind this definition of $z_{ij}$, we get exactly \eqref{Plucker_naive}.
It is also easily verifiable that $z_{ij}(gp)=\det(p_{11})z_{ij}(g)$. This $P$-equivariance means that $z_{ij}$ are fit to be the sections of ${\cal L}^{\vee}$, not ${\cal L}$.

Note that all of $z_{ij}$ cannot vanish simultaneously, and so they give a map
\bea Gr(2,V)\to \BB{P}^5,~~~[g]\to [z_{12},z_{13},z_{14},z_{23},z_{24},z_{34}],\nn\eea
i.e. the $z_{ij}$'s are now the homogeneous coordinates of $\BB{P}^5$. This map is also an embedding and realises the Grassmannian as a projective variety.

\smallskip

We now also phrase \eqref{Plucker_rel} as a representation theoretic problem, so we can generalise it later.
Note that a product $z_{ij}z_{kl}$ can be written as
\bea z_{ij}z_{kl}=\bra e^{\vee}_i\wedge e^{\vee}_j|g|e_1\wedge e_2\ket\bra e^{\vee}_k\wedge e^{\vee}_l|g|e_1\wedge e_2\ket=
\bra e^{\vee}_i\wedge e^{\vee}_j\otimes e^{\vee}_k\wedge e^{\vee}_l|g|e_1\wedge e_2\otimes e_1\wedge e_2\ket.\nn\eea
The factor $|e_1\wedge e_2\otimes e_1\wedge e_2\ket$ lives in $\wedge^2V\otimes^s\wedge^2V$, while
$\bra e^{\vee}_i\wedge e^{\vee}_j\otimes e^{\vee}_k\wedge e^{\vee}_l|$ lives in $\wedge^2V^*\otimes\wedge^2V^*$ (of which only the symmetric part matters). Let us decompose the latter tensor product. In terms of the Young tableaux $\wedge^2V^*\simeq {\tiny\yng(1,1)}$, and
\bea {\tiny\yng(1,1)}\otimes^s{\tiny\yng(1,1)}={\tiny\yng(1,1,1,1)}\oplus {\tiny\yng(2,2)}.\label{YD_decom}\eea
The first summand is of dim 1, on which $Gl(V)$ acts by the determinant. The projection to this summand is
\bea e^{\vee}_i\wedge e^{\vee}_j\otimes^s e^{\vee}_k\wedge e^{\vee}_l\to e^{\vee}_i\wedge e^{\vee}_j\wedge e^{\vee}_k\wedge e^{\vee}_l
=\ep_{ijkl}e^{\vee}_1\wedge e^{\vee}_2\wedge e^{\vee}_3\wedge e^{\vee}_4.\nn\eea
But observe that the same summand in the decomposition of $e_1\wedge e_2\otimes^s e_1\wedge e_2$ will be zero. This shows that
\bea z_{i[j}z_{kl]}=
\bra e_i^{\vee}\wedge e^{\vee}_{[j}\otimes e^{\vee}_k\wedge e_{l]}^{\vee}|g|e_1\wedge e_2\otimes e_1\wedge e_2\ket=0.\nn\eea
Writing out $z_{i[j}z_{kl]}$ explicitly, we recover the Pl\"ucker relation \eqref{Plucker_rel}.

\smallskip

For $Gr(2,4)$, \eqref{Plucker_rel} happens to be the only relation, thus $Gr(2,4)$ is embedded in $\BB{P}^5$ as a quadric surface. If we run the same construction for, say $Gr(2,5)$, we will still take the highest weight $e_1\wedge e_2\in \wedge^2V$, where $V=\BB{C}^5$.
When considering the decomposition of $e_1\wedge e_2\otimes^s e_1\wedge e_2\in \wedge^2V\otimes^s\wedge^2V$, its component in the summand $\wedge^4V$ is still zero, leading to the Pl\"ucker relation $z_{i[j}z_{kl]}=0$, except now there will be 5 such relations since $\dim \wedge^4V=5$.
This shows that the Pl\"ucker relations are redundant, indeed, they satisfy
\bea z_{ij}z_{kl}z_{pq}\ep^{jklpq}=0.\nn\eea
But these five redundancies are not independent either, and one goes on to find relations of ... of relations.
This procedure is clearly untenable for more complicated groups, and especially when the tensor decomposition is not as straightforward as \eqref{YD_decom}.

In practice, when faced with a subvariety in $\BB{P}^9$ cut out by $z_{i[j}z_{kl]}=0$, one can restrict to open sets of $\BB{P}^9$, where one can solve the Pl\"ucker relations explicitly in terms of 6 independent variables. These variables would then be the local complex coordinates of $Gr(2,5)$.
This will also be the approach for EIII.

\subsection{Pl\"ucker coordinates on $E_6/Spin(10)\times_{\BB{Z}_4}U(1)$}
We will construct the Pl\"ucker coordinates for $E_6/Spin(10)\times_{\BB{Z}_4}U(1)$.
To avoid repetition, in coming sections we set
\bea G:=E_6,~~~H:=Spin(10)\times_{\BB{Z}_4}U(1),\nn\eea
and $B\subset G_{\BB{C}}$ be the standard Borel subgroup (see 21.3 \cite{humphreys1975linear}).
Let us first recall the equivalence
\bea G/H\simeq G_{\BB{C}}/P,\label{equal_quot}\eea
where $P$ is the standard parabolic subgroup containing (and uniquely determined by) $H$ and $B$. The equality of the two quotients follows from the Iwasawa decomposition. We will spend a few words on this point. Let $G_{\BB{C}}$ be a complex Lie group and $B$ its Borel subgroup, then the Iwasawa decomposition gives
\bea G_{\BB{C}}=GB~{\rm or}~BG\label{Iwasawa}\eea
which is a generalisation of the Gram-Schmidt procedure, see e.g. theorem 26.3 \cite{bump2013lie}. This means that any $\gc\in G_{\BB{C}}$ can be written as $\gc=gb$ or $b'g'$ with $g,g'\in G$ and $b,b'\in B$. As a special case, if $H\subset G$ is a Poisson-Lie subgroup, and $g\in H$, then it follows that $g'$ is in $H$ too, see \cite{lu1990}. From this, the subset $HB$ is a subgroup containing $H,B$ and hence the standard parabolic $P$. The equality of the quotient \eqref{equal_quot} becomes now clear. For later use we remark that if $P=HB$, then $P\cap G=H$.

We will also need a convenient way to characterise the standard parabolic subgroups (those that contain $B$).  The Bruhat decomposition (28.3 \cite{humphreys1975linear}) says
\bea G_{\BB{C}}=\bigsqcup_{\gs\in W}B\gs B\label{Bruhat}\eea
where $\gs$ runs over the Weyl group $W$. The Weyl group is generated by the set of involutions $S=\{s_i\}$, with each $s_i$ corresponding to a simple root of $G$. Let $I\subset S$, and $W_I$ be generated by $I$, and $P_I=BW_IB$ is a standard parabolic subgroup uniquely characterised by $I$ (see 29.2, 29.3 \cite{humphreys1975linear}).

\smallskip

Back to the Pl\"ucker coordinates.
To mirror the construction of the last section, we need to find a representation of $G_{\BB{C}}$, and a dim 1 subspace that is preserved by $P$.
We can take ${\bf 27}^*$ and $\Phi_0$ its highest weight.
For a brief review of these representations, see appendix \ref{sec_E6}. We set $V=\BB{C}^{10}$, $\Gd^{\pm}$ the right/left handed spin representation of $Spin(10)$, then ${\bf 27}$ and ${\bf 27}^*$ decompose as representations of $H$ as
\bea {\bf 27}\simeq V_2\oplus \Gd^+_{-1}\oplus \BB{C}_{-4},\nn\\
{\bf 27}^*\simeq V^*_{-2}\oplus \Gd^-_{+1}\oplus \BB{C}_{+4}\nn\eea
where the subscript $2,-1$ etc are the weights under $U(1)$.
The last $\BB{C}$ factor in ${\bf 27}^*$ is the highest weight. We show in lem.\ref{lem_stab_grp_e6} that the stability group of this subspace is $P$: the Borel part naturally preserves the highest weight, the $Spin(10)$ acts trivially while $U(1)$ acts with weight 4.

Consider the associated line bundle ${\cal L}$
\bea
  \begin{tikzpicture}
  \matrix (m) [matrix of math nodes, row sep=2em, column sep=2em]
    {  G_{\BB{C}}\times_PL_4 & L_4 \\
       X=G_{\BB{C}}/P &  \\ };
  \path[->,font=\scriptsize]
  (m-1-1) edge node[right] {$\pi$} (m-2-1)
  (m-1-2) edge (m-1-1);
 \end{tikzpicture}\nn\eea
 where the subscript $L_4$ is to denote the weight under the action of $\BB{C}^*\subset P$.
 In exactly the same manner we can write down the holomorphic sections of the dual bundle ${\cal L}^{\vee}$, presented as $P$-equivariant maps
 \bea z_{\Psi}(g)=\bra \Psi|g|\Phi_0\ket\label{Plucker_E6}\eea
 where $\Phi_0$ is the highest weight of ${\bf 27}^*$ and $\Psi\in {\bf 27}$. This means we have 27 sections labelled by ${\bf 27}$. Concretely, we let $g$ act on $\Phi_0$, the entries of the resulting vector are the $z_{\Psi}$. 

Since the vector $g\Phi_0\neq 0$, that is, the 27 sections do not vanish simultaneously, we get a map
\bea X\to \BB{P}^{26}.\label{embedding}\eea
\begin{theorem}\label{thm_EIII}
  The homogeneous space $X$=EIII is embedded in $\BB{C}P^{26}$ via the Pl\"ucker coordinates.
\end{theorem}
\begin{proof}(Sketch)
  The proof is quite standard, see e.g. Th\'eor\`eme 3 in \cite{Serre_BWB}. We give a quick overview here. The tangent space at $[\Phi_0]$ (the line represented by $\Phi_0$) is spanned by the Lie algebra action of $\Gd_-\subset\FR{e}_6$ (see sec.\ref{sec_E6} for the description of $\FR{e}_6$), which coincides with $T_eG_{\BB{C}}/P$, showing that the map is an immersion at $e$, hence an immersion everywhere since $X$ is homogeneous. The properness is automatic and so it remains to show the map is 1-1. This means checking that the stability group of the line $[\Phi_0]$ is $H$, for which we give a direct proof in lem.\ref{lem_stab_grp_e6}. But slightly more abstractly, the proof proceeds as follows. Working in $G_{\BB{C}}$, the stability group of $[\Phi_0]$ certainly contains $B$ and so is a standard parabolic group. This means $P$ has the form $P=P_I$ for some $I$, as we have just reviewed. That is to say, $P$ is of the form $\cup_{\gs\in W_I}B\gs B$, with $W_I=\bra s_i,\,i\in I\ket$. Let $\gl$ be the dominant weight of the representation ${\bf 27}^*$. That $s_i$ preserves $\Phi_0$ means $\gl(\ga_i)=0$, where $\ga_i$ is the simple root corresponding to $s_i$. These $\ga_i$'s must therefore be the simple roots of $\FR{so}(10)$ (and only these). This uniquely characterises $P$ as the parabolic group $HB$, since the latter contains the same set of $s_i$'s. Taking intersection $P\cap G=H$.
\end{proof}

But we would like to describe this embedding more explicitly, in particular, we want to describe the image as cut out by a number of equations--the Pl\"ucker relations.
\begin{theorem}\label{thm_cut_by_Plucker}
  The image of $X\hookrightarrow\BB{C}P^{26}$ is cut out by 27 quadratic Pl\"ucker equations \eqref{Plucker_E6}.

  Further, we can cover $\BB{C}P^{26}$ with 27 standard coordinate charts. On each of the opens, we can solve the Pl\"ucker relations in terms of 16 local complex coordinates.
\end{theorem}
\begin{proof}
  The proof goes along the same line as the $Gr(2,4)$ case. Consider the product
  \bea z_{\Psi_1}(g)z_{\Psi_2}(g)=\bra \Psi_1\otimes \Psi_2|g|\Phi_0\otimes \Phi_0\ket.\nn\eea
  The tensor product decomposes as (see table 48 of \cite{SLANSKY19811})
  \bea {\bf 27}\otimes{\bf 27}\simeq{\bf 27}^*_s\oplus {\bf 351}_a\oplus {\bf 351}_s'\label{tensor_plucker}\eea
  where the subscript $a,s$ denotes if the two copies of ${\bf 27}$ are anti-symmetric or symmetric in the two factors.

  The tensor product $\Phi_0\otimes \Phi_0\in{\bf 27}^*\otimes^s{\bf 27}^*$ clearly has no component in ${\bf 351}^*_a$ due to symmetry. Nor does it have any component in ${\bf 27}$: the $U(1)$ weight of $\Phi_0\otimes \Phi_0$ is $8$ and there is no such vector in ${\bf 27}$.
  Based on this vanishing,
  we project $\Psi_1\otimes \Psi_2$ to ${\bf 27}^*$ and get
  \bea \bra (\Psi_1\otimes\Psi_2)_{27^*}|g|\Phi_0\otimes \Phi_0\ket=0\nn\eea
  giving us 27 relations on $z_{\Psi}(g)$. These are the Pl\"ucker relations.

  In prop.\ref{prop_d_tensor}, we give details of how the projection $\Psi_1\otimes \Psi_2 \to (\Psi_1\otimes\Psi_2)_{27^*}$ works.
  The main point is that ${\bf 27}$ has a totally symmetric invariant tensor $d(\textrm{-},\textrm{-},\textrm{-}):~{\bf 27}^{\otimes3}\to \BB{C}$. We regard $d$ as a map of representations
  \bea d:\; {\bf 27}\otimes^s {\bf 27}\to {\bf 27}^*.\nn\eea
  We denote also with $d$ the cubic invariant for ${\bf 27}^*$, thus
  the Pl\"ucker relation is now
  \bea d(z(g),z(g),-)=0.\label{Plucker_E6}\eea
  These relations are massively redundant, since by a simple counting, $\dim_{\BB{C}}X=16$, and there should only be $10=27-1-16$ independent relations.
  We have yet to show that these are the only relations, i.e. whether \eqref{Plucker_E6} cuts out the image of \eqref{embedding} in $\BB{C}P^{26}$.
  It is conceivable that there can be a purely representation theoretic proof of this statement, in the same way that \eqref{Plucker_E6} is constructed from representation theory. But our limited knowledge in the representations of $E_6$ forces us to take a more pedestrian route, which would meander around the rest of this section.

  Tentatively, we denote with $Y$ the subvariety in $\BB{C}P^{26}$ cut out by the Pl\"ucker relations
  \bea Y=\{z\in\BB{C}P^{26}|d(z,z,\textrm{-})=0\}.\label{tentative}\eea
  The Pl\"ucker coordinates $[g]\to z(g)$ gives a map $X\to Y$.
  We will first prove that $Y$ is a connected smooth variety of dimension 16, by
  solving \eqref{Plucker_E6} and finding the local coordinate charts, this will be done in the next section. Clearly $X$ is embedded in $Y$ as a submanifold. That $X,Y$ have the same dimension and $Y$ being connected means that the embedding is actually surjective, by ex.2 sec.4 ch.1 of \cite{guillemin2010differential} (we regret having to mix algebraic and differential topological arguments). Thus we have proved $X=Y$.
\end{proof}

\subsection{Solving the Pl\"ucker Relations}\label{sec_StPR}
We denote vectors in ${\bf 27}$ as
\bea \Psi=
\begin{bmatrix}
  v \\ \psi \\ s
\end{bmatrix},~~~~\Psi_0=\begin{bmatrix}
  0 \\ 0 \\ 1
\end{bmatrix},~~~~v\in V_2,~~\psi\in\Gd_{-1}^+,~~s\in\BB{C}_{-4}.\label{Psi_0_main}\eea
For the dual ${\bf 27}^*$, we write
\bea \Phi=
\begin{bmatrix}
  u \\ \phi \\ t
\end{bmatrix},~~~~\Phi_0=\begin{bmatrix}
  0 \\ 0 \\ 1
\end{bmatrix},~~~~u\in V_{-2},~~\phi\in\Gd^-_1,~~t\in\BB{C}_4.\label{Phi_0_main}\eea
The projection of $\Psi_1\otimes\Psi_2$ to ${\bf 27}^*$ is
\bea \opn{proj}_{27^*}(\Psi_1\otimes \Psi_2)=d(\Psi_1,\Psi_2,\textrm{-})=
\begin{bmatrix}
  v_1s_2+v_2s_1-\frac{1}{\sqrt2}(\psi_1e^i\psi_2)e^i \\ -\frac{1}{\sqrt2}(v_1\psi_2+v_2\psi_1) \\ v_1\cdotp v_2 \\
\end{bmatrix}.\label{should_be_zero}\eea
We refer the reader to prop.\ref{prop_d_tensor} for more details.
\begin{notation}
The Pl\"ucker coordinates can be thought of as a $\BB{C}$-linear map from ${\bf 27}$ to $\Gc({\cal L}^{\vee})$, i.e. given $\Psi\in{\bf 27}$, its corresponding section is $z_{\Psi}(g)=\bra\Psi|g|\Phi_0\ket$, where as usual, we identify $P$-equivariant functions as sections.

But in what follows, we work with the conjugate of these sections instead. This is due to a historical blunder that the author made in early stage of the work, which means all the technical computations were done in the opposite convention. We apologise for the inconvenience.

This means the sections will be written as $\bra\Phi|g|\Psi_0\ket$, with $\Psi_0$ being the \emph{lowest} weight in ${\bf 27}$. As ${\bf 27}$ decomposes as $Spin(10)$ representations as $V\oplus \Gd^+\oplus\BB{C}$, we denote with
\bea v=\opn{proj}_V(g|\Psi_0\ket),~~~\psi=\opn{proj}_{\Gd^+}(g|\Psi_0\ket),~~~s=\opn{proj}_{\BB{C}}(g|\Psi_0\ket)\nn\eea
and treat these as the Pl\"ucker coordinates instead (i.e. these are actually anti-holomorphic).
\end{notation}
The 27 Pl\"ucker relations are read off from \eqref{should_be_zero} $d(z,z,\textrm{-})=0$:
\bea 2vs-\frac{1}{\sqrt2}(\psi e^i\psi)e^i&=&0,~~~\times 10\label{equations_1}\\
v\psi&=&0,~~~\times 16 \label{equations_2}\\
v\cdotp v&=&0,~~~\times 1,\label{equations_3}\eea
see sec.\ref{sec_CA} for the Clifford algebra and spinor notations.

Before we get the coordinate charts, we want to convince the reader in one open stratum that the equation system is self-consistent and lends itself to elegant solutions.
Let us work in the open set $\{s\neq 0\}$. Then we can solve \eqref{equations_1}
\bea v=\frac{1}{2\sqrt2s}(\psi e^i\psi)e^i.\nn\eea
This $v$ will automatically satisfy \eqref{equations_2}, \eqref{equations_3}. Indeed
\bea v\cdotp v=\frac{1}{8s^2}(\psi e^i\psi)(\psi e^i\psi)\stackrel{\eqref{Jacobi_E6}}{=}0,\nn\eea
solving \eqref{equations_3}, while \eqref{equations_2} holds for the same reason.
In the rest of this section, we will solve the equation system in other open sets.

We cover $\BB{C}P^{26}$ with the 27 standard open sets. On the open $s\neq 0$, we are done. On the open set where
\bea t_+=\sqrt2(v_9+iv_{10})\neq 0\nn\eea
we split the 10D Clifford algebra in 8D plus 2D. We adopt the representation of Clifford algebra as in \eqref{gamma_explicit}, which is friendly to our 8+2 split:
\bea &&e^i=
\begin{bmatrix}
  {\tt e}^a & 0 \\ 0 & {\tt e}^a
\end{bmatrix},~~~a=1,\cdots,8\nn\\
&&e^9+ie^{10}=
\begin{bmatrix}
  0 & 2\gc_8 \\ 0 & 0
\end{bmatrix},~~~e^9-ie^{10}=
\begin{bmatrix}
  0 & 0 \\ 2\gc_8 & 0
\end{bmatrix},\nn\\
&&C_{10}=\begin{bmatrix}
  0 & C_8 \\ C_8 & 0
\end{bmatrix},~~~\gc_{10}=\begin{bmatrix}
  \gc_8 & 0 \\ 0 & -\gc_8
\end{bmatrix},\nn\eea
where ${\tt e}^a$ are the representation of $e^a$ as generators of the dim 8 Clifford algebra, and $C_8$, $\gc_8$ are the dim 8 charge conjugation and chirality operator. The reader may verify these using the explicit representations \eqref{charge_conj} \eqref{chirality}.
We split likewise $\Gd^{\pm}$ of 10D as 8D spinors
\bea \Gd^+_{10}=
\begin{bmatrix}
  \Gd^+_8 \\ \Gd^-_8
\end{bmatrix},~~~\Gd^-_{10}=
\begin{bmatrix}
  \Gd^-_8 \\ \Gd^+_8
\end{bmatrix},~~~\psi\in\Gd^+_{10},~~\psi=
\begin{bmatrix}
  \xi \\ \eta
\end{bmatrix}.\nn\eea
We rewrite the Pl\"ucker relations as
\bea &&2s{\tt u}-\sqrt2(\xi {\tt e}^a\eta){\tt e}^a=0,\nn\\
&&st_++(\eta\eta)=0,\nn\\
&&st_--(\xi\xi)=0,\nn\\
&&
\begin{bmatrix}
  {\tt u} & \frac{1}{\sqrt2}t_-\gc_8 \\ \frac{1}{\sqrt2}t_+\gc_8 & {\tt u}
\end{bmatrix}
\begin{bmatrix}
  \xi \\ \eta
\end{bmatrix}=0,\nn\\
&&{\tt u}\cdotp{\tt u}+\frac12t_+t_-=0,\label{6_eq}\eea
where the typewriter font ${\tt u}$ means the first eight components of $v$.

Since $t_+\neq 0$, we solve equation 2, 5 and the lower one of 4
\bea \xi=-\sqrt2t_+^{-1}{\tt u}\eta,~~~s=-t_+^{-1}(\eta\eta),~~~t_-=-2t_+^{-1}{\tt u}\cdotp{\tt u}.\nn\eea
We plug the solution to the other equations, e.g. for equation 1
\bea 2\big(-t_+^{-1}(\eta\eta)\big){\tt u}+2t_+^{-1}(\eta{\tt u}e^a\eta)e^a=-2t_+^{-1}(\eta\eta){\tt u}+t_+^{-1}(\eta\{{\tt u},e^a\}\eta)e^a=0\checkmark.\nn\eea
The remaining equations can also be checked. Thus in this open set, we have seventeen projective coordinates $t_+,\eta,{\tt u}$.

The treatment on the open set $\{t_-\neq0\}$ is entirely the same, and we get
\bea s=t_-^{-1}(\xi\xi),~~~t_+=-2t_-^{-1}{\tt u}\cdotp{\tt u},~~~\eta=\sqrt2t_-^{-1}{\tt u}\xi\nn\eea
and we have $t_-,\xi,{\tt u}$ as coordinates.

When $t_+=t_-=0$ i.e. $v_9=v_{10}=0$, we need only apply a $Spin(10)$ transformation. Concretely, the Pl\"ucker relations are Spin(10) equivariant, and so we can rotate another component of $v$ to be its 10th component. To see how this works, referring to def.\ref{def_Spin}, we set
\bea \Go=\frac{1}{\sqrt2}(1+e^ae^{10}),~~~\Go e^a\Go^{-1}=-e^{10},~~\Go e^{10}\Go^{-1}=e^a\nn\eea
i.e. $\Go$ is a Spin lift of the $SO(10)$ that rotates the $e^a$ and $e^{10}$ plane.
We apply $\Go$ to all three equations \eqref{equations_1}, \eqref{equations_2} and \eqref{equations_3}. In particular
\bea \Go\eqref{equations_2}=(\Go v\Go^{-1})\Go\psi=0.\nn\eea
We denote $\hat v=\Go v\Go^{-1}$, and $\hat\psi=\Go\psi$, where
\bea \hat\psi=
\begin{bmatrix}
  \hat\xi \\ \hat\eta
\end{bmatrix}=\frac{1}{\sqrt2}
\begin{bmatrix}
  \xi+i{\tt e}^a\eta \\ \eta+i{\tt e}^a\xi
\end{bmatrix}.\nn\eea
Thus if the corresponding $\hat t_+$ or $\hat t_-$ is nonzero, we can solve the Pl\"ucker relations as above, merely adding a hat on everything.

\subsection{Solving the Pl\"ucker Relations: at $X_{\infty}$}\label{sec_StPR_2}
The coordinate charts so far can cover all the loci where $s$ or $v\neq 0$. The remaining case to tackle is the locus
\bea X_{\infty}=\{s=v=0\}\cap X.\nn\eea
Let us first analyse what this locus is. Assuming $s=0$, then we need to solve
\bea (\psi e^i\psi)=0,~~i=1,\cdots,10.\nn\eea
As preparation, we give a name to these spinors.
\begin{definition}\label{def_pure_spinor}
  In dim 10, a spinor $\psi\in \Gd^+$ is called \emph{pure} if
  $\psi\otimes\psi$ has zero projection to $V^*$ in the decomposition
  \bea\Gd^+\otimes\Gd^+\simeq V^*\oplus \wedge^3V^*\oplus \wedge^5_{-i}V^*.\nn\eea
\end{definition}
Concretely $\psi$ being pure means $(\psi e^i\psi)=0$, for all $i$. Less obvious is the relation $(\psi e^{ijk}\psi)=0$, which is due to symmetry $(\psi e^{ijk}\psi)=(\psi e^{kji}\psi)=-(\psi e^{ijk}\psi)$.
\begin{lemma}\label{lem_pure_spin_J}
  Assume $\psi$ is pure and normalised: $\bra \psi,\psi\ket=1$, then it defines a complex structure
  \bea v\to Jv=\frac{i}{2}\bra \psi,[e^i,v]\psi\ket e^i.\label{cplx_struc}\eea

  Further, the following elements of $\wedge^5V$ are (anti-)holomorphic
  \bea \Go=\frac{\sqrt2}{8\cdotp5!}\bra\psi,e^{i_1\cdots i_5}C\psi^*\ket e^{i_1}\wedge\cdots\wedge e^{i_5},
  ~~\bar \Go=\frac{\sqrt2}{8\cdotp5!}(\psi,e^{i_1\cdots i_5}\psi)e^{i_1}\wedge\cdots\wedge e^{i_5},~~\frac{1}{5!}\Go^{i_1\cdots i_5}\bar\Go^{i_1\cdots i_5}=1.\nn\eea
  We also have the equality
  \bea \frac{1}{4!}\bra\psi e^{ijkl}\psi\ket e^i\wedge\cdots\wedge e^l=-\frac12J\wedge J,~~~~J=\frac{i}{2}\bra\psi,[e^i,e^j]\psi\ket e^i\wedge e^j,\nn\eea
  where we took the liberty of identifying $J$ as in $\wedge^2V$.
\end{lemma}
Do pay attention that in the formula above we used both pairings $\bra\cdots\ket$ and $(\cdots)$, see def.\ref{def_spin_action} for the difference.
\begin{remark}\label{rmk_adapted_spin_rep}
  All of the relations given here are proved by applying the Fierz relations. The process is a bit tedious. So instead, we indicate how one can understand these relations intuitively by introducing an adapted spin representation (though these relations all hold independent of any particular choice of representations).

  Assume we are given a complex structure $J$ in 10 dimension, set
\bea \Gd^+=\Go^{{\rm ev},0},~~~\Gd^-=\Go^{{\rm od},0},\nn\eea
where $\Go^{p,q}$ are the $(p,q)$ form on $\BB{C}^5\simeq \BB{R}^{10}$. The Clifford action is realised as: pick complex basis $e^a,e^{\bar a}$ for $\BB{C}^5$. Then given $\psi\in\Gd^{\pm}$, we have
\bea e^a\psi=\sqrt2 e^a\wedge\psi,~~~e^{\bar a}\psi=\sqrt2\iota_a\psi,~~~e^{a_1\cdots a_p}\sim 2^{p/2}dx^{a_1}\wedge \cdots\wedge dx^{a_p}\label{Clifford_J}\eea
i.e. the Clifford action is either wedging a form $\psi$ with $e^a$ or contracting with $e^a$.

Now we set $\psi=1$, the 0-form, then \eqref{cplx_struc} gives
\bea &&Je^a=\frac{i}{2}\bra\psi[e^i,e^a]\psi\ket e^i=i([\iota_b,e^a])_0 e^b=ie^a,\nn\\
&&Je^{\bar a}=i([e^b,\iota_a])_0 e^{\bar b}=-ie^{\bar a},\nn\eea
where $(\cdots)_0$ means extracting the 0-form component. This shows that with $\psi=1$ (which is pure), we get the complex structure we started with.

Now the statement that $\Go$ be holomorphic should be clear, since any $u$ with $Ju=-iu$ would annihilate $\psi_0$. In the same way, one can understand the relation involving the 4-form.
Finally if $\psi=1$, then $e^aC\psi^*=0$ for all $e^a$, so $C\psi^*$ is proportional to the holomorphic 5-form $e^1\wedge\cdots\wedge e^5$, in fact, from the way we defined $\Go$, $C\psi^*=\Go$.
\end{remark}
We show next that $X_{\infty}$ is another nice homogeneous space
\begin{lemma}
 The locus $X_{\infty}$ is parametrised by a pure spinor $\psi$ and it is the homogeneous space
\bea X_{\infty}\simeq SO(10)/U(5).\nn\eea
\end{lemma}
\begin{proof}
The locus $X_{\infty}$ is parametrised by a pure spinor $\psi$, which defines a complex structure as in lem.\ref{lem_pure_spin_J}. The space of complex structures in $\BB{R}^{10}$ is precisely described by the homogeneous space stated. The scaling of $\psi$ does not matter as $\psi$ are themselves the homogeneous coordinates of $\BB{C}P^{26}$.

Conversely, given a complex structure $J$, we can define a pure spinor in the way we have described in the last remark.
\end{proof}

Now we will find open sets that can cover $X_{\infty}$ and solve the Pl\"ucker coordinates there.
Set $\psi_0$ to be pure, with $\bra\psi_0\psi_0\ket=1$. Set further
\bea &&f=\bra \psi_0\psi\ket\in \BB{C},~x=(\psi_0e^i\psi)e^i\in V,\nn\\
&&\tilde K=\frac12\bra \psi_0e^{ij}\psi\ket e^i\wedge e^j\in\wedge^2V,\nn\\
&&K=\frac12(\bra \psi_0e^{ij}\psi\ket-f\bra \psi_0e^{ij}\psi_0\ket) e^i\wedge e^j\in\wedge^2V,\label{set_further}\eea
where we have freely identified $V$ with $V^*$ with the complexified inner-product on $V=\BB{C}^{10}$.
Here the point is that if $\psi$ were to be a pure spinor, then there must be some pure spinor $\psi_0$, such that $f=\bra\psi_0\psi\ket\neq 0$. In such case, we can express $\psi$ in terms of $\psi_0$ and $f,x,K$. We can also divide by $f$ to go from homogeneous to inhomogeneous coordinates.
\begin{lemma}\label{lem_technical}
We can express $\psi$ using $f,x,K,\psi_0$
\bea \psi=\psi_0f-\frac{1}{8}K\psi_0+\frac12xC\psi_0^*.\label{psi_psi_0}\eea
If $J$ is the complex structure determined by $\psi_0$ (Eq.\ref{cplx_struc}), then $x$ is of type (0,1) with respect to $J$ i.e. $Jx=-ix$, while $K$ is of type (2,0).

We record also the relation
\bea (\psi e^i\psi)e^i=2fx-Kx-\frac{1}{8\sqrt2}\bar\Go^{ipqrs}K^{pq}K^{rs}e^i.\nn\eea
\end{lemma}
\begin{proof}
  For the first statement, we consider $\psi=\psi\bra \psi_0\psi_0\ket $, and apply \eqref{Fierz_10_key} (set $\psi_2=\psi$, $\psi_1=\psi_0$ and $\eta=C\psi_0^*$)
  \bea \psi=\psi\bra \psi_0\psi_0\ket=-\frac14\psi_0\bra\psi_0\psi\ket-\frac{1}{8}e^{ij}\psi_0\bra\psi_0e^{ij}\psi\ket+\frac12e^iC\psi_0^*(\psi_0e^i\psi)\nn\eea
  We name the tensors $f,x,\tilde K$ as in \eqref{set_further}, thus
  \bea \psi=-\frac{1}{4}\psi_0f-\frac{1}{8}\tilde K^{ij}e^{ij}\psi_0+\frac{1}{2}x^ie^iC\psi_0^*.\nn\eea

  Now we prove the properties of $f,x,\tilde K$.
  We first show that $u\in V$ is (0,1) (i.e. $Ju=-iu$) \emph{iff} $u\psi_0=0$.\footnote{This would be tautology using the adapted spinor representation as in rmk.\ref{rmk_adapted_spin_rep}. But we choose not to go this route so as to demonstrate that none of the relations depend on any concrete choice of spin representations.}
  Assume $u$ is (0,1), then
  \bea -iu\psi_0=Ju\psi_0=\frac{i}{2}\bra \psi_0[e^i,u]\psi_0\ket e^i\psi_0=i\bra\psi_0(-\hcancel{ue^i}+u^i)\psi_0\ket e^i\psi_0=iu\psi_0,\nn\eea
  where the crossed out term vanishes due to \eqref{Jacobi_E6} (set $\psi_1=\psi_3=\psi_0$ and $\psi_2=uC\psi_0^*$). From this we get $u\psi_0=0$. For the converse, assume $u\psi_0=0$
  \bea Ju=\frac{i}{2}\bra \psi_0[e^i,u]\psi_0\ket e^i=
  i\bra \psi_0(e^iu-u^i)\psi_0\ket e^i=-iu.\nn\eea

  Now we show that $x$ is $(0,1)$, i.e. we need to check if $x\psi_0=0$:
  \bea x\psi_0=(\psi_0 e^i\psi)e^i\psi_0\stackrel{\eqref{Jacobi_E6}}{=}-\frac12(\psi_0e^i\psi_0)e^i\psi=0\checkmark.\nn\eea

  For the tensor $\tilde K$, let $v$ be such that $Jv=iv$, then $vC\psi^*_0=0$ for a similar reason as above. Then
  \bea \tilde K^{ij}v^j=\bra\psi_0 e^{ij}\psi\ket e^i v^j=\bra\psi_0(-\hcancel{ve^i}+v^i)\psi\ket e^i=vf,\nn\eea
  this shows that the combination
  \bea K=\tilde K+ifJ\nn\eea
  is of type (2,0), so it makes sense to use $K$ instead of $\tilde K$. Now we need evaluate $J^{ij}e^{ij}\psi_0$ (again effortlessly computed using the adapted spinor representation).
  We compute
  \bea J^{ij}e^{ij}\psi_0&=&i\bra\psi_0e^{ij}\psi_0\ket e^{ij}\psi_0\stackrel{\eqref{Fierz_10_key}}{\to}\nn\\
  &=&-2i\bra\psi_0\psi_0\ket \psi_0-8i\psi_0\bra\psi_0\psi_0\ket=-10i\psi_0.\nn\eea
  With this, we get to
  \bea \psi=\psi_0f-\frac{1}{8}K^{ij}e^{ij}\psi_0+\frac{1}{2}x^ie^iC\psi_0^*,\nn\eea
  which is \eqref{psi_psi_0}.

  Now we can compute $(\psi e^i\psi)$ in terms of $f,K,x$. We will use \eqref{psi_psi_0} to replace $\psi$
  \bea (\psi e^i\psi)e^i&=&f^2\hcancel{(\psi_0e^i\psi_0)}e^i-\frac{1}{64}(\psi_0Ke^iK\psi_0)e^i+\frac{1}{4}\hcancel{(C\psi^*_0,xe^ixC\psi_0^*)}e^i\nn\\
  &&-\frac{f}{4}\hcancel{(\psi_0e^iK\psi_0)}e^i+f(\psi_0e^ixC\psi_0^*)e^i+\frac{1}{8}(\psi_0Ke^ixC\psi_0^*)e^i,\nn\eea
  where the three terms vanish due to purity, see def.\ref{def_pure_spinor} and the comment immediately thereafter. For the remaining terms
  \bea (\psi e^i\psi)e^i=-\frac{1}{64}(\psi_0Ke^iK\psi_0)e^i+f(\psi_0e^ixC\psi_0^*)e^i+\frac{1}{8}(\psi_0Ke^ixC\psi_0^*)e^i,\nn\eea
  we define the (0,5)-form as in lem.\ref{lem_pure_spin_J}
  \bea \bar\Go^{i_1\cdots i_5}=\frac{1}{4\sqrt2}(\psi_0e^{i_1\cdots i_5}\psi_0),\nn\eea
  and we get
  \bea (\psi e^i\psi)e^i
  =-\frac{1}{8\sqrt2}\bar\Go^{pqirs}K^{pq}e^iK^{rs}+f(\psi_0e^ixC\psi_0^*)e^i-\frac{1}{8}\bra \psi_0xe^i K\psi_0\ket e^i.\nn\eea
  For the middle term, we apply the Clifford relation to move $x$ to the left and kill $\psi_0$ as we have shown above, and we get $2fx$.
  For the last term, we move $x$ to the right to kill $\psi_0$
  \bea \bra \psi_0xe^iK\psi_0\ket&=&2x^i\hcancel{\bra \psi_0K\psi_0\ket}-\bra \psi_0e^ixK\psi_0\ket
  =-\bra \psi_0e^i(-4Kx)^je^j\psi_0\ket\nn\\
  &=&-4i(JKx)^i+4(Kx)^i=8(Kx)^i,\nn\eea
  where we crossed a term since $K$ is (2,0), and we also used the relation $[x,K]=-4(Kx)^ie^i$. Altogether
  \bea (\psi e^i\psi)e^i=-\frac{1}{8\sqrt2}\bar\Go^{ipqrs}K^{pq}K^{rs}e^i+2fx-Kx.\nn\eea
  We have finished the proof.
\end{proof}
The point of this technical lemma is that when solving the Pl\"ucker relations, we need to solve some spinor equations, but by this lemma, we express $\psi$ as some tensors and we can turn the Pl\"ucker relation into tensor equations, which we can solve.
\begin{proposition}
Fix a normalised pure spinor $\psi_0$, in the open set $\{f=\bra\psi_0\psi\ket\neq 0\}$, the Pl\"ucker relation \eqref{equations_1}-\eqref{equations_3} can be solved in terms of a (2,0) tensor $K$, a (0,1) vector $\bar u$ and two scalars $s,f$ (17 in total):
 \bea
 \Psi&=&[v;\psi;s],\nn\\
 v&=&-\frac{1}{2}f^{-1}K\bar u+\bar u,\nn\\
 \psi&=&\psi_0f-\frac{1}{8}K\psi_0+\frac{1}{\sqrt2}f^{-1}\big(s\bar u+\frac{1}{32}\bar\Go^{ipqrs}K^{pq}K^{rs}e^i\big)C\psi_0^*.\nn\eea
\end{proposition}
\begin{proof}
Using lem.\ref{lem_technical}, we write the Pl\"ucker relations in terms of $x,f,K$
\bea I:&&2vs-\frac{1}{\sqrt2}(-\frac{1}{8\sqrt2}\bar\Go^{ipqrs}K^{pq}K^{rs}e^i+2fx-Kx)=0,\nn\\
II:&&v(\psi_0f-\frac{1}{8}K\psi_0+\frac12xC\psi_0^*)=0,\label{used_middle}\\
III:&&v^2=0.\nn\eea
We split $v$ into its (1,0) part $u$ and its (0,1) part $\bar u$ (but $u,\bar u$ are not conjugate of each other). The first equation above splits into
\bea I_1:&&2us+\frac{1}{\sqrt2}Kx=0,\nn\\
 I_2:&&2\bar us+\frac{1}{16}\bar\Go^{ipqrs}K^{pq}K^{rs}e^i-\sqrt2fx=0.\nn\eea
 To analyse \eqref{used_middle}, one can again use the Fierz relation to express it as tensors as we did for $\psi$ itself. But for a bit of variation, we recall that when represented as forms $\psi_0=1$ and $C\psi_0^*=\Go$, so \eqref{used_middle} decomposes according to form degrees
\bea II_1:&&u\psi_0f-\frac{1}{8}\bar uK\psi_0=0,~\To~uf+\frac{1}{2}K\bar u=0,\nn\\
II_3:&&-\frac{1}{8}uK\psi_0+\frac12\bar uxC\psi_0^*=0,~\To~-\frac{\sqrt2}{2}u\wedge K+\iota_{\bar u}\iota_x\Go=0,\nn\\
II_5:&&\frac12uxC\psi_0^*=0~\To~u\cdotp x=0,\nn\eea
where $II_{1,3,5}$ are respectively 1- 3- and 5-forms and factors of $\sqrt 2$ come from \eqref{Clifford_J}.
We solve $x$ from $I_2$ and $u$ from $II_1$
\bea &&x=\sqrt2f^{-1}\bar us+\frac{1}{16\sqrt2}f^{-1}\bar\Go^{ipqrs}K^{pq}K^{rs}e^i\nn\\
&&u=-\frac{1}{2}f^{-1}K\bar u.\nn\eea
These actually solve all the other equations. To start, the equation $III:\;v^2=u\cdotp\bar u=0$ because $K$ is anti-symmetric.
Checking $II_5:\;u\cdotp x=0$ is more interesting
\bea u\cdotp x=\sqrt2f^{-1}\hcancel{u\cdotp\bar u}s+\frac{1}{16\sqrt2}f^{-1}\bar\Go^{ipqrs}K^{pq}K^{rs}u^i
=-\frac{1}{32\sqrt2}f^{-2}\bar\Go^{ipqrs}K^{pq}K^{rs}K^{ij}\bar u^j.\nn\eea
To see that the rhs is zero, we use complex basis since $K$ is of type (2,0).
Since $K$ is now an anti-symmetric $5\times5$ matrix, it is of rank $0,2,4$. Whichever rank $K$ has, the contraction $\ep_{a_1a_2a_3a_4a_5}K^{a_2a_3}K^{a_4a_5}$ is in the kernel of $K$, hence $\ep_{a_1a_2a_3a_4a_5}K^{a_0a_1}K^{a_2a_3}K^{a_4a_5}=0$.

We check now $I_1$
\bea 0\stackrel{?}{=}2us+\frac{1}{\sqrt2}(\sqrt2f^{-1}K\bar us+\hcancel{\frac{1}{16\sqrt2}f^{-1}\bar\Go^{ipqrs}K^{pq}K^{rs}K^{ji}e^j})
=(2u+f^{-1}K\bar u)s~\checkmark \nn\eea
where we crossed out the second term based on exactly the same argument.

It remains to check $II_3$,
\bea II_3=
-\frac{\sqrt2}{8f}\iota_{\bar u}(K\wedge K)+\iota_{\bar u}(\sqrt2sf^{-1}\iota_{\bar u}\Go+\frac{1}{16\sqrt2f}\bar\Go^{ipqrs}K^{pq}K^{rs}\iota_i\Go).\nn\eea
The middle term is zero, to manipulate the rest, it is easier to switch to holomorphic basis and treat $\Go,\bar\Go$ as (anti-)holomorphic volume form. The last term is then
\bea \bar\Go^{ipqrs}K^{pq}K^{rs}\iota_i\Go=4**(K\wedge K)=4(-1)^{4(5-4)}K\wedge K.\nn\eea
So all equations are satisfied.
\end{proof}

To summarise, we found dense open sets in sec.\ref{sec_StPR} and \ref{sec_StPR_2} that cover the entire subvariety $Y$ in \eqref{tentative}, and we have found explicit local holomorphic coordinates in each open. We conclude that $Y$ is a connected smooth subvariety of $\BB{C}P^{26}$ and we finally finish the proof of thm.\ref{thm_EIII}, that is, the homogeneous space $X=E_6/Spin(10)\times_{\BB{Z}_4}U(1)$ coincides with $Y$.

\section{The Octonion Projective Plane}
For the division algebras $\BB{D}=\BB{R},\BB{C},\BB{H}$, we can define the projective spaces as the equivalence class
\bea &&\BB{D}P^n=\{[x_0,\cdots,x_n]|x_i\in \BB{D}\}/\sim,\nn\\
&& [x_0,\cdots,x_n]\sim [x_0r,\cdots,x_nr],~~r\in \BB{D}^{\times}.\nn\eea
But if $\BB{D}$ is not associative, the equivalence relation above fails to be transitive and the construction fails.
However for small $n=1,2$, there are still ways to extend the construction to $\BB{O}$. We take $n=2$, and review three ways this can be done, see sec 3.4 of \cite{Baez:2001dm}.

First we turn to another description of $\BB{C}P^2$. Let $P$ be a rank 3 hermitian matrix of rank 1, satisfying
\bea PP=\frac12\{P,P\}=P.\nn\eea
This means that $P$ is a projection operator, whose image is of dim 1. The collection of such $P$ gives the set of 1D subspaces in $\BB{C}^3$, which is $\BB{C}P^2$.

Based on this, we take $\BB{D}=\BB{O}$, we can directly generalise the last construction, let $\BB{D}P^2$ be the set of hermitian octonion matrices $P$, with rank 1, and
\bea P\star P:=\frac12\{P,P\}=P.\nn\eea
The rank 1 condition of $P$ is imposed with $\Tr[P]=1$. Both equations above will turn out to be $F_4$-invariant. In fact, a second way of realising $\BB{O}P^2$ is via the coset $F_4/Spin(9)$.

A third way is to simply use the inhomogeneous coordinates. For $\BB{C}P^2$, we can cover it with three open charts: $U_i=\{x_i\neq 0\}$, $i=0,1,2$. On, say, $U_1$, we can use $(x_0/x_1,x_2/x_1)$ as the inhomogeneous coordinates. On the intersection $U_i\cap U_j$, one can easily get the transition function relating the two sets of inhomogeneous coordinates. This procedure would not seem generalisable to $\BB{O}P^2$, since we have three open sets and three intersections, and the transition function may not fulfil the cocycle condition. But it turns out that it is possible to satisfy the cocycle conditions, and thereby describe $\BB{O}P^2$ by local octonion coordinate charts.

\subsection{The Coset Construction}
Before going to EIII, it is helpful to warm up with the real version
\bea \frac{F_4}{Spin(9)}\simeq \BB{O}P^2.\label{coset_F4}\eea
Our starting point is the dim 26 real representation of $F_4$, which we review in sec.\ref{sec_F4}. It is convenient to regard ${\bf 26}$ as a subspace of the representation ${\bf 27}$ of $E_6$ ($F_4$ is a subgroup of $E_6$): ${\bf 27}$ splits as ${\bf 26}\oplus{\bf 1}$.
We use the parametrisation
\bea \Psi=\begin{bmatrix}
           u\oplus \frac{-is}{\sqrt2}e^{10} \\ \psi \\ s
          \end{bmatrix}\in{\bf 26},~~~\Psi_{\emptyset}=\frac{\sqrt2}{\sqrt3}\begin{bmatrix}
                                   ie^{10} \\ 0 \\ \frac{1}{\sqrt2}
                                  \end{bmatrix}\in{\bf 1}\nn\eea
that is, even though $\Psi$ is written as 27 dimensional, its $10^{th}$ and $27^{th}$ components are tied.
It would appear more sensible to work with ${\bf 26}$ since it is an irreducible representation, yet
we found it easier to take ${\bf 26}$ and ${\bf 1}$ together. In fact adding the trivial factor ${\bf 1}$ makes the calculation neater.

As representations of $E_6$, the ${\bf 27}$ is not real, but as representations of $F_4$, there is an isomorphism $\gs:\,{\bf 27}^*\simeq {\bf 27}$, given in lem.\ref{lem_26_real}. Concretely
\bea \Psi=\begin{bmatrix}
           u\oplus \frac{it}{\sqrt2}e^{10} \\ \psi \\ s
          \end{bmatrix}\in{\bf 26}\oplus{\bf 1}\nn\eea
is real if
\bea u\in\BB{R}^9,~~s,t\in \BB{R},~~ie^{10}C_{10}\psi^*=\psi.\nn\eea

We choose a specific vector
\bea \Psi_o=
\begin{bmatrix}
  -\sqrt2ie^{10} \\ 0 \\ 0
\end{bmatrix}.\label{Psi_o}\eea
The stability group of $\Psi_o$ is $Spin(9)$ and so its orbit space is
\bea \opn{Orb}(\Psi_o)=\frac{F_4}{Spin(9)}.\label{orbit_Psi_0}\eea
To show this, note that by adding to $\Psi_o$ a vector proportional to the singlet $\Psi_{\emptyset}$, we can get $[0;0;1]$ i.e. \eqref{Psi_0_main} and so the orbit space can be treated identically as in the $E_6$ case in lem.\ref{lem_stab_grp_e6}.

\subsection{The Jordan Algebra Construction}
All is straightforward enough so far, what is not clear is the relation to octonions.

We will find a map between our Pl\"ucker coordinates $\Psi$ and $3\times3$ octonion matrices such that the Jordan product corresponds to the `product' \eqref{Plucker_E6}.
\begin{definition}({\bf $\diamond$ product})
  Let $\Psi_{1,2}\in {\bf 26}\oplus{\bf 1}$ be real, then the cubic invariant of $E_6$ defines a symmetric product $\diamond$
  by sending $\Psi_1,\Psi_2$ to $d(\Psi_1,\Psi_2,\textrm{-})$ which is in ${\bf 26}^*\oplus{\bf 1}^*$, followed by the isomorphism $\gs$ that leads back to ${\bf 26}\oplus{\bf 1}$.
\end{definition}
Concretely, we make a split $10=9+1$ and write a real $\Psi\in {\bf 26}\oplus{\bf 1}$ as
\bea \Psi=
\begin{bmatrix}
  v \\ \psi \\ s
\end{bmatrix},~~~v=(u,\frac{1}{\sqrt2}it),\nn\eea
with $u,s,t$ real, $\psi$ satisfying $ie^{10}C_{10}\psi^*=\psi$.
The product reads
\bea
  \begin{bmatrix}
    u_1+\frac{i}{\sqrt2}t_1e^{10} \\ \psi_1 \\ s_1
  \end{bmatrix}\diamond\begin{bmatrix}
    u_2+\frac{i}{\sqrt2}t_2e^{10} \\ \psi_2 \\ s_2
  \end{bmatrix}&\stackrel{\ref{prop_d_tensor}}{=}&\gs\begin{bmatrix}
    u_{(1}s_{2)}-\frac{1}{\sqrt2}(\psi_1e^i\psi_2)e^i+\frac{i}{\sqrt2}e^{10}(t_{(1}s_{2)}-\bra\psi_1\psi_2\ket) \\
    -\frac{1}{\sqrt2}u_{(1}\psi_{2)}-\frac{i}{2}e^{10}t_{(1}\psi_{2)} \\
     u_1\cdotp u_2-\frac12t_1t_2 \\
  \end{bmatrix}\nn\\
  &=&\begin{bmatrix}
    u_{(1}s_{2)}-\frac{1}{\sqrt2}(\psi_1e^i\psi_2)e^i -\frac{i}{\sqrt2}e^{10}(t_{(1}s_{2)}-\bra\psi_1\psi_2\ket) \\
    -\frac{i}{\sqrt2}e^{10}u_{(1}\psi_{2)}+\frac{1}{2}t_{(1}\psi_{2)} \\
     u_1\cdotp u_2-\frac12t_1t_2 \\
\end{bmatrix},\label{diamond_product}\eea
where the $i$ index goes over $1,\cdots,9$ and $u_1,u_2\in\BB{R}^9$.

Written in the 8+2 split
\bea v=({\tt u},\frac{1}{\sqrt2}r,\frac{1}{\sqrt2}it),~~~\psi=
\begin{bmatrix}
  \xi \\ \eta
\end{bmatrix},\nn\eea
with ${\tt u},r,t$ real and $\xi,\eta$ Majorana. Then we have the product
\bea
  &&\begin{bmatrix}
    {\tt u}_1+\frac{1}{2\sqrt2}t^{(+}_1e^{-)} \\ \psi_1 \\ s_1
  \end{bmatrix}\diamond\begin{bmatrix}
    {\tt u}_2+\frac{1}{2\sqrt2}t^{(+}_2e^{-)} \\ \psi_2 \\ s_2
  \end{bmatrix}\nn\\
  &=&\begin{bmatrix}
    {\tt u}_{(1}s_{2)}-\frac{1}{\sqrt2}(\xi_{(1}{\tt e}^i\eta_{2)}){\tt e}^i+\frac1{2\sqrt2}(t^-_{(1}s_{2)}-2(\xi_1\xi_2))e^-+\frac1{2\sqrt2}(t^+_{(1}s_{2)}+2(\eta_1\eta_2))e^+ \\
    -\frac{1}{\sqrt2}\begin{bmatrix}
                               {\tt u}_{(1}\eta_{2)} \\ {\tt u}_{(1}\xi_{2)}
                              \end{bmatrix}+\frac{1}{2}\begin{bmatrix}
                               -t^+_{(1}\xi_{2)} \\ t^-_{(1}\eta_{2)}
                              \end{bmatrix} \\
     {\tt u}_1\cdotp {\tt u}_2+\frac14t^{(+}_1t^{-)}_2 \\
 \end{bmatrix},\nn\eea
where $t^{\pm}=r\mp t$, $e^{\pm}=e^9\pm ie^{10}$

Next, we map $\Psi$ to a $3\times3$ Hermitian octonion matrix.
\begin{theorem}\label{thm_star_diamond}
 We make a split $10=8+2$ and write a real $\Psi\in {\bf 26}\oplus{\bf 1}$ as
\bea \Psi=
\begin{bmatrix}
  v \\ \psi \\ s
\end{bmatrix},~~~v=({\tt u},\frac{1}{\sqrt2}r,\frac{1}{\sqrt2}it),~~\psi=
\begin{bmatrix}
  \xi \\ \eta
\end{bmatrix},\nn\eea
with ${\tt u},r,s,t$ real, $\xi,\eta$ Majorana. We assign to $\Psi$ the Hermitian octonion matrix
\bea J(\Psi)=\begin{bmatrix}
  \frac{1}{2}(r-s) & \frac{1}{\sqrt2}\xi & \frac{1}{\sqrt2}{\tt u} \\
  \frac{1}{\sqrt2}\bar\xi & \frac12(s-t) & \frac{1}{\sqrt2}\eta \\
  \frac{1}{\sqrt2}\bar {\tt u} & \frac{1}{\sqrt2}\bar\eta & -\frac{1}{2}(r+s) \\
\end{bmatrix},\label{arrangement}\eea
where we identify $\xi,\eta,{\tt u}$ with octonions as in sec.\ref{sec_g2fCA}.
Then the Jordan product $\star$ and the $\diamond$ product are related as
\bea J(\Psi_1)\star J(\Psi_2)=-\frac12J(\Psi_1\diamond\Psi_2)+\frac14\bra \Psi_1,\Psi_2\ket {\bf 1}_3.\label{master_eq}\eea
Furthermore, the trace
\bea \Psi\to \Tr[J(\Psi)]\nn\eea
is $F_4$-invariant.
\end{theorem}
\noindent The proof is a direct computation, given in sec.\ref{sec_EJA}.
We get an easy but important corollary
\begin{corollary}
  The coset $F_4/Spin(9)$ is the octonion projective space $\BB{O}P^2$, which can be described by the equation
  \bea P\star P=P,~~~\Tr[P]=1.\label{PP=P}\eea
  The correspondence reads
  \bea \opn{Orb}(\Psi_o)\ni \Psi\to P=J(\Psi).\nn\eea
\end{corollary}
\begin{proof}
The orbit of $\Psi_o$ is $F_4/Spin(9)$, see \eqref{orbit_Psi_0} and the discussion there.
A direct calculation shows that $\Psi_o$ satisfies
\bea \Psi_o\diamond\Psi_o=-2\Psi_o-2\sqrt3\Psi_{\emptyset},~~~\bra\Psi_o,\Psi_o\ket=2,\nn\eea
then \eqref{master_eq} gives
\bea J(\Psi_o)\star J(\Psi_o)=-\frac12J(\Psi_o\diamond\Psi_o)+\frac14\bra\Psi_o,\Psi_o\ket{\bf 1}=J(\Psi_o+\sqrt3\Psi_{\emptyset})+\frac12{\bf 1}=J(\Psi_o),\nn\eea
which gives \eqref{PP=P}. In fact, $J(\Psi_o)=\opn{diag}[0,1,0]$, the above relation is immediate.
If this relation is to hold for $\Psi_o$, then it holds for all $\Psi$ in the $F_4$-orbit of $\Psi_o$.
The statement about the trace works in the same way.
At the same time, \eqref{PP=P} describes $\BB{O}P^2$ as we reviewed in the beginning of this section.

On the other hand, we need to show that any $P$ satisfying \eqref{PP=P} is in the image of $J(\opn{Orb}(\Psi_o))$. The logic of the proof is similar to that of thm.\ref{thm_cut_by_Plucker}: comparing the tangent space at $[1]\in F_4/Spin(9)$ and at $J(\Psi_o)$ shows that the map $J$ is a local diffeomorphism. It remains to prove that \eqref{PP=P} cuts out a smooth connected manifold. For this, we need to solve \eqref{PP=P} and get the local coordinate charts, which we do in the remaining part of this section.
\end{proof}

\smallskip

To get the local octonion coordinate charts, we spell out the equations \eqref{PP=P} first
\bea \frac{1}{2}(r-t-2)\xi+\frac{1}{\sqrt2}{\tt u}\eta=0,\nn\\
-\frac12(r+t+2)\eta+\frac{1}{\sqrt2}{\tt u}\xi=0,\nn\\
-(s+1){\tt u}+\frac{1}{\sqrt2}(\xi{\tt e}^i\eta){\tt e}^i=0,\nn\\
\frac12(r-s)(r-s-2)+(\xi\xi)+{\tt u}\cdotp{\tt u}=0,\nn\\
\frac12(s-t)(s-t-2)+(\xi\xi)+(\eta\eta)=0,\nn\\
\frac12(r+s)(r+s+2)+(\eta\eta)+{\tt u}\cdotp{\tt u}=0,\nn\\
s+t=-2.\label{7_eq}\eea
The last four equations give
\bea(\xi\xi)=(r-s)(s+1),~~~(\eta\eta)=-(r + s)(s + 1),~~~{\tt u}\cdotp{\tt u}=-\frac12(r+s)(r-s).\nn\eea
At this point, we digress and define the so called `Veronese coordinates' following \cite{salzmann2011compact}
\begin{definition}\label{def_Veronese}(Sec 16.1 \cite{salzmann2011compact})
  Let $(\vec x;\vec\gl)=(x_1,x_2,x_3;\gl_1,\gl_2,\gl_3)\in\BB{O}^3\times\BB{R}^3$. We say $(\vec x,\vec\gl)$ is a Veronese vector if
  \bea \gl_1\bar x_1=x_2\star x_3,~~|x_1|^2=\gl_2\gl_3,~~~\textrm{cyc perm}~(1,2,3).\nn\eea
\end{definition}
\begin{proposition}(Sec 16.8 \cite{salzmann2011compact})
  The Veronese condition plus $\sum\gl_i=1$ is equivalent to \eqref{PP=P}.
\end{proposition}
\begin{proof}
  Let us recall that Clifford actions such as ${\tt u}\xi$ can be written in terms of octonion multiplications, see prop.\ref{prop_oct_mult_acatar}. We can write the first three equations of \eqref{7_eq} as
  \bea &&{\tt u}\star\bar\eta=-\frac{1}{\sqrt2}(r+s)\xi,~~\bar\xi\star{\tt u}=\frac{1}{\sqrt2}(r-s)\eta,~~\xi\star\eta=\sqrt{2}(s+1){\tt u},\nn\eea
  where we have used $s+t=-2$.
  We also recall that for $\BB{O}P^2$, the spinors $\xi,\eta$ are real, and so $(\xi\xi)=|\xi|^2$, $(\eta\eta)=|\eta|^2$, thus
  \bea |\xi|^2=(r-s)(s+1),~~|\eta|^2=-(r+s)(s+1),~~|{\tt u}|^2=-\frac12(r+s)(r-s).\nn\eea
  So we can set
  \bea (x_1,x_2,x_3;\gl_1,\gl_2,\gl_3)=(2^{-1/2}{\tt u},2^{-1/2}\bar\eta,2^{-1/2}\bar\xi,s+1,2^{-1}(r-s),-2^{-1}(r+s)),\nn\eea
  the vector is Veronese and $\sum_i\gl_i=1$. The converse is now also clear.
 \end{proof}

Now we can get the local octonion coordinates on the open sets $U_i$.
\begin{proposition}
  On the open set $U_i=\{\gl_i\neq 0\}$, we set the local octonion coordinates as
  \bea U_i:~~~(a_i,b_i)=(\gl_i^{-1}x_{i+1},\gl_i^{-1}x_{i+2}).\nn\eea
  On the intersection $U_i\cap U_{i+1}$, the transition functions read
  \bea a_{i+1}=\bar b_i^{-1},~~~b_{i+1}=b_i^{-1}\star \bar a_i\nn\eea
  which satisfy the cocycle condition.
\end{proposition}
\begin{proof}
The transition functions follow from a direct computation. For example, setting $i=1$, on $U_1\cap U_2$, we have $\gl_1\neq 0$ and $\gl_2\neq 0$, implying that $x_3\neq 0$ and hence invertible,
\bea &&b_2=\frac{x_1}{\gl_2}=\frac{\bar x_3}{\gl_1}\star\frac{\bar x_2}{\gl_1}\frac{\gl_1}{\gl_2}
=\frac{\bar x_3}{\gl_1}\star\frac{\bar x_2}{\gl_1}\frac{\gl^2_1}{|x_3|^2}=\left(\frac{x_3}{\gl_1}\right)^{-1}\star\frac{\bar x_2}{\gl_1}=b_1^{-1}\star\bar a_1,\nn\\
&&a_2=\frac{x_3}{\gl_2}=\left(\frac{\bar x_3}{\gl_1}\right)^{-1}=\bar b_1^{-1}.\nn\eea

On the intersection $U_1\cap U_2\cap U_3$, one would naively expect the cocycle condition to fail, due to the lack of associativity. But let us recall that the octonions form an alternating algebra, in that its associator $(a\star b)\star c-a\star(b\star c)$ is totally anti-symmetric in $a,b,c$. Also if any one of $a,b,c$ is real, the associativity holds too (see prop.\ref{prop_alternating}). Thus we can safely change the brackets in e.g. $(\bar a\star a)\star c=\bar a\star(a\star c)$ and
$a^{-1}\star(a\star c)=c$ if we remember $\bar a=2\re a-a$ and $a^{-1}\sim \bar a$.
\end{proof}
\begin{remark}
Since the transition function given here exactly mimics that of the usual complex projective plane, one can rightfully regard $F_4/Spin(9)$ as the octonion version of the projective plane.

Note also that the transition functions involves multiplying $b_i^{-1}$ for $\bar b_i^{-1}$ from the appropriate side. This is chosen so that we can still satisfy the cocycle conditions on triple intersections. This `explains' why we can only have $\BB{O}P^1$ or $\BB{O}P^2$: essentially we run out of ways of fiddling with the position of $b_i^{-1}$ or $\bar b_i^{-1}$.
\end{remark}

\subsection{Complex Octonion Projective Plane}
We can apply the treatment of the previous section and relate the coset $E_6/Spin(10)\times_{\BB{Z}_4}U(1)$ with the $\diamond$ product.\footnote{We thank the referee for pointing two other related formulations of EIII using either the Jordan pair \cite{truini1982e6} or the Jordan triple system \cite{Koecher1969AnEA}. Especially, with the Jordan pairs one has the so called inner ideals, with which one can define complex octonion points, lines and planes and their incidence relations neatly.} But this is not straightforward, since for $E_6$, there is no relation between ${\bf 27}$ and its conjugate, so we cannot define the trace.
Furthermore, the $d$ tensor gives a map ${\bf 27}\otimes{\bf 27}\to {\bf 27}^*$ and we cannot turn ${\bf 27}^*$ back to ${\bf 27}$, and so the $d$ tensor does not give us an algebra. However, let us formally mimic the same procedure and see what we get.

Parameterise $\Psi\in{\bf 27}$ as
  \bea \Psi=
\begin{bmatrix}
  v \\ \psi \\ s
\end{bmatrix},~~~v=({\tt u},\frac{1}{\sqrt2}r,\frac{1}{\sqrt2}it),~~\psi=
\begin{bmatrix}
  \xi \\ \eta
\end{bmatrix},\nn\eea
where we no longer impose $ie^{10}C\psi^*=\psi$. Assign to $\Psi$ a complex octonion matrix
\bea J(\Psi)=\begin{bmatrix}
  \frac{1}{2}(r+t) & \frac{1}{\sqrt2}\xi & \frac{1}{\sqrt2}{\tt u} \\
  \frac{1}{\sqrt2}\bar\xi & s & \frac{1}{\sqrt2}\eta \\
  \frac{1}{\sqrt2}\bar {\tt u} & \frac{1}{\sqrt2}\bar\eta & -\frac{1}{2}(r-t) \\
\end{bmatrix}\nn\eea
where now $\xi,\eta,{\tt u}$ are complexified octonions and the bar operation is $\BB{C}$-linear.
\begin{proposition}
There is an $E_6$-invariant determinant of the complex octonion matrix $J(\Psi)$
\bea \det J(\Psi)=-\frac12s(v\cdotp v)+\frac{\sqrt2}{2}(\eta{\tt u}\xi)+\frac14(\xi\xi)t^+-\frac14(\eta\eta)t^-.\nn\eea
If $\Psi$ is in the $E_6$-orbit of $\Psi_0$ in \eqref{Psi_0_main}, then $J(\Psi)$ satisfy
\bea J\star J=(t+s)J=\Tr[J]J\label{mean_unclear}\eea
\end{proposition}
\noindent We stress that the trace here merely mimics the $F_4$ case, and is not $E_6$ invariant.
\begin{proof}
 We should clearly define the determinant using the $d$ tensor
 \bea \det J(\Psi)=-\frac16d(\Psi,\Psi,\Psi).\nn\eea
 The computation uses prop.\ref{prop_d_tensor} and is straightforward. We give below the same formula entirely in terms of octonions.

 As for the second relation, \eqref{6_eq} gives
 \bea (J\star J)_{12}&=&\frac{1}{2\sqrt2}(r+t+2s)\xi+\frac12{\tt u}\eta=\frac{1}{\sqrt2}(t+s)\xi,\nn\\
 (J\star J)_{23}&=&\frac{1}{2\sqrt2}(2s-r+t)\eta+\frac12{\tt u}\xi=\frac{1}{\sqrt2}(t+s)\xi,\nn\\
 (J\star J)_{13}&=&\frac{1}{\sqrt2}t\xi+\frac12(\xi {\tt e}^i\eta){\tt e}^i=\frac{1}{\sqrt2}(t+s)\xi,\nn\\
 (J\star J)_{11}&=&\frac14(r+t)^2+\frac12(\xi\xi)+\frac12{\tt u}\cdotp{\tt u}=\frac12(s+t)(r+t),\nn\\
 (J\star J)_{22}&=&s^2+\frac12(\xi\xi)+\frac12(\eta\eta)=s^2+st=(s+t)s,\nn\\
 (J\star J)_{33}&=&\frac14(r-t)^2+\frac12(\eta\eta)+\frac12{\tt u}\cdotp{\tt u}=-\frac12(s+t)(r-t),\nn\eea
  where we used the Pl\"ucker relation written in 8D terms (see sec.\ref{sec_StPR})
  \bea &&{\tt u}\eta=-2^{-1/2}(r-t)\xi,~~{\tt u}\xi=2^{-1/2}(r+t)\eta,~~~(\xi {\tt e}^i\eta){\tt e}^i=2^{1/2}s{\tt u},\nn\\
  &&(\xi\xi)=s(r+t),~~(\eta\eta)=-s(r-t),~~{\tt u}\cdotp{\tt u}=-\frac12(r^2-t^2).\label{Veronese_C}\eea
  From these we get \eqref{mean_unclear}.
\end{proof}
If one writes a Hermitian complex octonion matrix in generic terms, then we have the standard determinant for $3\times 3$ octonion matrices
\bea \det\begin{bmatrix}
          \ga & \bar z & \bar y \\ z & \gb & x \\ y & \bar x & \gc
         \end{bmatrix}=\ga\gb\gc-(\ga x\star\bar x+\gb y\star\bar y+\gc z\star\bar z)+2z\cdotp(x\star y).\nn\eea

More importantly, the relations \eqref{Veronese_C} is actually the complexification of the Veronese condition def.\ref{def_Veronese}. We have
\begin{proposition}
  Setting
  \bea (\vec x;\vec\gl)=(2^{-1/2}{\tt u},2^{-1/2}\bar\eta,2^{-1/2}\bar\xi;s,2^{-1}(t+r),2^{-1}(t-r)),\nn\eea
  then the Pl\"ucker relations \eqref{equations_1}, \eqref{equations_2}, \eqref{equations_3} hold iff $(\vec x;\vec\gl)$ is complexfied Veronese:
  \bea x_i\star x_{i+1}=\gl_{i+2}\bar x_{i+2},~~~x_i\star \bar x_i=\gl_{i+1}\gl_{i+2}.\nn\eea
\end{proposition}
\begin{remark}
We remind the reader that for complexified octonions the bar operation is $\BB{C}$-linear and so $a\star \bar a\in\BB{C}$, while if bar should conjugate both factors $\BB{C}\otimes\BB{O}$, then $a\star \bar a\notin\BB{R}$ or $\BB{C}$ in general.

This has the consequence that the open sets $U_i=\{\gl_i\neq 0\}$ no longer cover the whole space: setting $\gl_1=\gl_2=\gl_3=0$ would give $x_i\star \bar x_i=0$, but this no longer forces all $x_i=0$. The authors of \cite{Corradetti:2022sot} suggested modifying the bar operation so that it conjugates both factors in $\BB{C}\otimes\BB{O}$. This does avoid the problem just mentioned, but it is not clear how to relate their construction with the coset $G/H$. Still, their construction is quite interesting.
\end{remark}

Rather, to interpret \eqref{mean_unclear}, we note that the trace $\Tr[J]$ is not $E_6$ invariant, but $F_4$ invariant. This hints at studying the decomposition of the $E_6$ homogeneous space as a collection of $F_4$-homogeneous spaces.

\subsection{Decomposition as $F_4$ homogeneous spaces}
Sumihiro proved in \cite{Sumihiro} the following result. Let $G$ be a connected linear algebraic group, and $X$ a normal quasi-projective variety $X$ with a $G$-action. Then $X$ has a $G$-stable covering $X=\cup_iU_i$, with $U_i$ quasi-projective and the $G$-action on $U_i$ is linear, that is, $G$ acts linearly on $\BB{P}^{n_i}$, into which $U_i$ embeds. Thus we can decompose $X$ into strata of subvarieties with a $G$-action on each stratum. This is the `equivariant completion'.

In \cite{Ahiezer}, Ahiezer classified the case with $X=X_0\cup X_{\infty}$ with $X_0$ an open $G$-orbit and $X_{\infty}$ of co-dimension 1, see table 2 loc. cit. We take an example from the same table to illustrate the idea of completion.
\begin{example}
Consider the subvariety $X_0=\{z_1^2+\cdots+z_n^2=1\}\subset\BB{A}^n$. Scaling up the $z_i$'s to `go to infinity' we get the equation $X_{\infty}=\{z_1^2+\cdots+z_n^2=0\}$. We compactify  $X_0$ by adding $X_{\infty}$: consider the subvariety $X=\{z_1^2+\cdots+z_n^2=z_0^2\}\subset \BB{P}^n$. Over the affine patch $z_0\neq 0$, we get $X_0$, intersecting $X$ with $\{z_0=0\}$ we get $X_{\infty}$. Note $X_0$ is the orbit under $O(n,\BB{C})$ of $(1,0,\cdots,0)$, while $X_{\infty}$ is the orbit of $(1,i,0,\cdots,0)$.
\end{example}
We are interested in the $G=F_4$ row of the cited table.
\begin{proposition}\label{prop_punch_line}
  The homogeneous space $X$ has a two strata decomposition $X=Y_0\cup Y_{\infty}$
  where $Y_0$ is an open $F_4$-orbit and is a complexification of $\BB{O}P^2$.
  The remaining stratum $Y_{\infty}$ has co-dimension 1.
\end{proposition}
\noindent This result says that $X$ is the \emph{projective completion} of $(\BB{O}P^2)_{\BB{C}}$, which is our punch line \eqref{punch_line}.
\begin{proof}\footnote{I saw the idea of considering equivariant completion in a post in some maths forum, but as I was typing up the manuscript, I could no longer find that post, I apologise for not giving credit to the right person.}
Let $R$ be the 27 dimensional representation of $E_6$, whose projectivisation is $\BB{C}P^{26}$. Recall that $X$ is the $E_6$-orbit of $\Psi_0$ in $\BB{C}P^{26}$. To consider its decomposition as $F_4$-orbits, we decompose ${\bf 27}={\bf 26}\oplus{\bf 1}$. Here the component of $\Psi\in{\rm Orb}(\Psi_0)$ along ${\bf 1}$ will play the role of $z_0$ in the example above.

Returning to the equation \eqref{mean_unclear}, which is equivalent to $\Psi$ satisfying the Pl\"ucker relation.
The trace is an $F_4$-invariant, in fact, $s+t=\Tr[J(\Psi)]=\sqrt3\bra\Psi,\Psi_{\emptyset}\ket$, where $\bra,\ket$ is the $F_4$-invariant pairing and $\Psi_{\emptyset}$ is the generator of ${\bf 1}$ in ${\bf 26}\oplus{\bf 1}$.

In the open set where $s+t\neq 0$, we define
\bea P=\frac{1}{s+t}J,~~~P\star P=P,~~~\Tr[P]=1\nn\eea
which is formally \eqref{PP=P}, except the $P$ here is complexified compared to \eqref{PP=P}, i.e. we obtain a complexification
\bea \BB{O}P^2 \to (\BB{O}P^2)_{\BB{C}},\nn\eea
and the rhs is an open $F_4$-orbit.

The remaining stratum has $s+t=0$, i.e. $\Psi$ has zero component along ${\bf 1}$: it is the intersection of $X$ with $\BB{C}P^{25}$, which is the projectivisation of ${\bf 26}$.
\end{proof}

\section{Summary and Outlook}
This paper's main technical result is the Pl\"ucker coordinates that embed the space EIII in $\BB{C}P^{26}$. These Pl\"ucker coordinates satisfy some quadratic relations, in the same way as the classical Grassmannian case.

The central conceptual result is the precise sense in which EIII can be interpreted as the complex octonion projective plane: the naive complexification of $\BB{O}P^2$ needs completion by adding another $F_4$-orbit at infinity. This structural result is new, to the best of my knowledge.

The $E_6$ group has an important role in theoretical physics, e.g. the grand unified theories (GUT): its gauge group $E_6$ contains a possible GUT gauge group $SO(10)$. So when $E_6$ is (spontaneously) broken to this subgroup, the vev of the scalar fields responsible for the breaking would live in a manifold closely related to our quotient space $E_6/Spin(10)\times_{\BB{Z}_4}U(1)$. The geometry of this quotient directly dictates the low energy physics, so any means that gives us a better understanding of its geometry should be helpful in some way.

We thank the referee for pointing out further applications in supergravity.
Though we have only been considering the compact $E_6$, part of our results can be readily generalised to the non-compact form $E_{6(-14)}$ and the quotient symmetric space $E_{6(-14)}/Spin(10)\times_{\BB{Z}_4}U(1)$. This space appears as the (extended) scalar manifold of ${\cal N}=10$ supergravity in $D=2+1$ (see table 11 of \cite{Ferrara2008SymmetricSI}). Furthermore the same quotient appears as the moduli space of the `non-BPS $Z=0$’ critical points of the black hole effective
potential in the `magical exceptional' ${\cal N}=2$ Maxwell-Einstein (ungauged) supergravity \cite{GUNAYDIN198372} in
$D=3+1$. In this context, $E_{6(-14)}/Spin(10)\times_{\BB{Z}_4}U(1)$ is characterised
as a proper submanifold of the special K\"ahler symmetric vector multiplets’ scalar manifold
$E_{7(-25)}/E_6\times U(1)$ (see table 2 of \cite{Ferrara2008SymmetricSI}), which in turn is the Riemannian, totally non-compact form of the
Hermitian symmetric space EVII.

The referee also brought \cite{truini1982e6} to our attention. The authors of that paper defined a Jordan pair, i.e. a pair of $3\times 3$ complex hermitian octonion matrices with some idempotency consditions. There are three types of such idempotents, type I consists of pairs $(\Psi,\Phi)\in{\bf 27}\times{\bf 27}^*$, such that both $\Psi$ and $\Phi$ satisfy the Pl\"ucker relations and $\bra\Psi,\Phi\ket=1$.
Upon normalising $|\Psi|^2=|\Phi|^2=1$, then $\Psi=\Phi^*$, that is, one can also describe EIII as normalised the Jordan pairs of type I.

Take two normalised pairs of idempotents $x_i=(\Psi_i,\Phi_i)$, $i=1,2$. Then they define both points $p(x_i)$ and lines $\ell(x_i)$.
The points are just the line $[\Psi_i]$ (upon projectivisation) and so concurs with the description of the current paper. As for the lines, one says that $p(x_1)$ is incident to $\ell(x_2)$ iff $(\Psi_1\diamond\Psi_2)\diamond\Phi_2=\Psi_1$, where $\diamond$ is the product introduced in prop.\ref{prop_d_tensor}.

It was shown in \cite{truini1982e6} that one can have two points in certain configuration (called connected) such that more than one line can pass both points. Dually, two lines may also intersect at more than one point. Thus, despite the name octonion projective plane, the geometry does not satisfy the axioms of projective planes. The curious incidence/intersection relations also lead to novel quantum logic. The formulation in \cite{truini1982e6} was in terms of octonion matrices, while our Pl\"ucker coordinate (in particular, manifestly $E_6$ covariant) approach seems to simplify many calculations. We hope to explore this further in some future work.

\appendix

\section{Spinors and Exceptional Groups}
We follow the treatment of exceptional Lie algebras in \cite{Adams1996lectures}. This treatment relies on manipulations of Clifford algebra and spinors extensively, we encourage the reader to take a look at chapter 2,3,4 in \cite{Adams1996lectures}. Here we just gather the minimal amount of material for setting the notations.

\subsection{Clifford Algebra}\label{sec_CA}
The main reference here would be the classic work \cite{ATIYAH19643}.
Let $V=\BB{R}^m$ with $\{e^i,\;i=1,\cdots, m\}$ as its standard basis. The Clifford algebra is defined as the quotient of the tensor algebra $T(V)$
\bea Cl(V)=T(V)/J,~~~J=\bra v\otimes v-(v\cdotp v)1|v\in V\ket,\nn\eea
where $\cdotp$ is the standard inner-product of $V$.
We will omit the tensor symbol $\otimes$ in the following. Put simply, $Cl(V)$ consists of words generated by the alphabet $\{e^i\}$, with the relation $e^i e^j+e^je^i=2\gd^{ij}$.
Here our convention is slightly different from \cite{ATIYAH19643}, which has $e^i e^j+e^je^i=-2\gd^{ij}$, this entails a small modification of def.\ref{def_Spin} below.

In view of the fact that the relation reduces word length, we have the isomorphism of vector spaces
\bea Cl(V)\simeq \oplus_{0\leq p\leq m}\wedge^pV.\nn\eea
Since the relation $J$ preserves the tensor grading mod 2, we let $Cl_0(V)$ and $Cl_1(V)$ be the even and odd part of $Cl(V)$. Let $\ga:\,Cl(V)\to Cl(V)$ be the automorphism that adds a negative sign to $Cl_1$ but leaves $Cl_0$ untouched. We also define an anti-automorphism $()^t$, for $x=e^{i_1}\cdots e^{i_k}$ a monomial, set
\bea x^t=e^{i_k}\cdots e^{i_1}.\label{t_ope}\eea
Lastly let $Cl^{\times}(V)$ be the set of invertible elements.

Given $x\in Cl^{\times}(V)$, we have a twisted adjoint action
\bea \rho(x):\; Cl(V)\to Cl(V),~~~y\to \ga(x)yx^{-1}.\label{twisted_adj}\eea
Define the Clifford group $\Gc(V)$ as the set
\bea  \Gc(V)=\{x\in Cl^{\times}(V)|\rho(x)V\subset V\}.\nn\eea
Set also \bea N(x)=x\ga(x^t).\nn\eea
Prop.3.8 of \cite{ATIYAH19643} shows that if $x\in \Gc(V)$, then $N(x)\in\BB{R}^{\times}$.
\begin{definition}\label{def_Spin}
The Pin group is defined as
\bea Pin(V)=\{x\in \Gc(V)|N(x)=\pm1\}.\nn\eea
The Spin group is defined as $Spin(V)=Pin(V)\cap Cl_0(V)$. It is a double cover of $SO(V)$.
\end{definition}
\begin{remark}
  One can be even more specific in describing $Pin(V)$ or $Spin(V)$. The former consists of monomials of $x=x_1\cdots x_k$ with $x_i$ of unit length, and the latter are those monomials with $k$ even.
  Indeed, such monomials fulfil the definition above. On the other hand, the transformation $\ga(x)vx^{-1}=\ga(x)vx^t$ is given by reflecting $v$ across the plane perpendicular to $x_i$. It is well-known that any $O(V)$ can be written as a succession of such reflections.
\end{remark}

We will only deal with $V$ of even dimension, so we let $m=2n$.
\begin{example}\label{ex_Gd_and_gamma}(Compare with table 1 in \cite{ATIYAH19643})

The Dirac representation $\Gd$ of the Clifford algebra has dimension $2^n$. If we set $\Gd=\BB{C}^2\otimes \cdots\otimes \BB{C}^2$, then the generators $\{e^i\}$ of the Clifford algebra have the matrix representation
  \bea &&e_1=\gs_1\otimes 1\otimes \cdots \otimes 1,\nn\\
  &&e_2=\gs_2\otimes 1\otimes \cdots \otimes 1,\nn\\
  &&e_3=\gs_3\otimes \gs_1\otimes 1\otimes \cdots \otimes 1,\nn\\
  &&e_4=\gs_3\otimes \gs_2\otimes 1\otimes \cdots \otimes 1,\nn\\
  &&\cdots \nn\\
  &&e_{2n-1}=\gs_3\otimes \cdots\otimes \gs_3 \otimes \gs_1,\nn\\
  &&e_{2n}=\gs_3\otimes \cdots\otimes \gs_3 \otimes \gs_2,\label{gamma_explicit}\eea
  where there are $n$ tensor factors, and $\gs_{1,2,3}$ are the Pauli matrices
  \bea \gs_1=
  \begin{bmatrix}
    0 & 1 \\ 1 & 0
  \end{bmatrix},~~\gs_2=
  \begin{bmatrix}
    0 & -i \\ i & 0
  \end{bmatrix},~~\gs_3=
  \begin{bmatrix}
    1 & 0 \\ 0 & -1
  \end{bmatrix}.\nn\eea
\end{example}
But $\Gd$ is reducible as a representation of $Spin(V)$. One can decompose it into $\Gd=\Gd^+\oplus \Gd^-$ of dimension $2^{n-1}$ each.
In physics literature, $\Gd^{\pm}$ are denoted as right/left handed spinors, and are characterised by their eigen-value under the `chirality operator'
\bea \gc=(-i)^ne_1e_2\cdots e_{2n},~~~\gc^2=1,~~\gc\Gd^{\pm}=\pm\Gd^{\pm}.\label{chirality}\eea
The choice of $(-i)^n$ will become clear later. The important property of $\gc$ is
\bea \{\gc,e^i\}=0\nn\eea
thus an odd number of $e^i$ would flip the chirality, however $\FR{spin}(V)$ has the form $e^ie^j$ and so preserves the chirality.

The conjugate (or dual) representation of $\Gd^{\pm}$ are not new ones, but rather
\bea (\Gd^{\pm})^*\simeq \Gd^{\pm},~~n~\opn{even},\nn\\
(\Gd^{\pm})^*\simeq \Gd^{\mp},~~n~\opn{odd},\label{charge_conj}\eea
see prop.4.3 in \cite{Adams1996lectures}. For us, it is important to know the isomorphism above explicitly. We define the `charge conjugation' matrices $C$, according to the physics terminology
\bea n~\opn{even}:&&C=e_2e_4\cdots e_{2n},\nn\\
n~\opn{odd}:&&C=(-1)^{\frac{n-1}{2}}e_1e_3\cdots e_{2n-1}.\nn\eea
The $C$ matrices are chosen so that
\bea C^{-1}x^TC=x^t.\label{t_T}\eea
The notation here is:
we have identified $x\in Cl(V)$ with the matrix representing it, so the left $T$ is the matrix transposition, while the right $t$ is the anti-automorphism \eqref{t_ope}.

Note that all the $e^i$'s are represented as Hermitian matrices and so transposition is the same as conjugation. So we can define the charge conjugation that explicitly realises the isomorphism \eqref{charge_conj}
\bea n~\opn{even}:&&\Gd^{\pm}\ni  \phi \to C\phi^*\in \Gd^{\pm}\label{duality_even}\\
n~\opn{odd}:&&\Gd^{\pm}\ni  \phi \to C\phi^*\in \Gd^{\mp}\label{duality_odd}.\eea
where $*$ is the complex conjugation. In particular for $\dim=4k+2$, the charge conjugation flips chirality, while for $\dim=4k$, the chirality is preserved. Thus one may well ask if one can impose the reality for $\Gd^+$ and $\Gd^-$, i.e. if there exists $\psi\in\Gd^+$ with $C\psi^*=\psi$. This turns out only possible if $\dim=8k$.
For this work 8D spinors play a lead role and the fact that we can have real left/right handers of real dimension 8, and hence identical to the dimension of the vector representation of $SO(8)$, is crucial. This is known as the \emph{triality} and responsible for the octonion multiplication rules.

As a vector space $\FR{spin}(V)=\FR{so}(V)$ is generated abstractly by $\{x^{pq}|x^{pq}=-x^{qp}\}$, with the Lie bracket
\bea[x^{pq},x^{rs}]=\gd^{qr}x^{ps}+\textrm{anti-sym in}\;[pq][rs].\nn\eea
The following definition shows how the $\FR{spin}$ action covers the $\FR{so}$ action.
\begin{definition}\label{def_spin_action}
  The action of the Lie algebra $x^{pq}=e^p\wedge e^q\in\FR{spin}(V)$ on $\Gd$ or $\Gd^{\pm}$ is via the Clifford multiplication
  \bea x^{pq}\circ\psi=\frac14e^{[p}e^{q]}\psi.\nn\eea
  On $V$, we deduce from \eqref{twisted_adj} that $x^{pq}$ acts as the commutator of Clifford algebra
  \bea x^{pq}\circ v=[\frac14e^{[p}e^{q]},v].\nn\eea
\end{definition}

We use $\bra\;,\;\ket$ to denote the Hermitian inner-product for $\Gd^+$ and $\Gd^-$, i.e. $\bra\psi,\phi\ket=\psi^{\dag}\phi$ for $\psi,\phi$ of the same chirality. The dualities \eqref{duality_even}, \eqref{duality_odd} provides another non-degenerate bi-linear pairing
\bea n~\opn{even}:&&\Gd^{\pm}\otimes \Gd^{\pm} \ni  \phi\otimes \psi  \to (\phi,\psi):=\phi^TC\psi\label{pairing_even}\\
n~\opn{odd}:&&\Gd^{\pm}\otimes \Gd^{\mp} \ni \phi\otimes \psi  \to (\phi,\psi):=\phi^TC\psi.\label{pairing_odd}\eea
We will often omit the commas and write $(\phi\psi)$ or $\bra\phi\psi\ket$ instead.

\subsection{Tensor Product of Spinors and Fierz Identities}%, say, $\Gd^{\pm}\otimes \Gd^{\pm}$
The tensor product of two spin representations can be decomposed into direct sums of tensor representations of $\FR{spin}(V)$, see theorem 4.6 \cite{Adams1996lectures}. Here we list some decompositions relevant for us
\bea \dim=16:&&\Gd^{\pm}\otimes\Gd^{\pm}={\color{blue}\wedge^8_{\pm}V^*}\oplus {\color{red}\wedge^6V^*}\oplus {\color{blue}\wedge^4V^*}\oplus {\color{red}\wedge^2V^*}\oplus{\color{blue}\bf 1},\nn\\
&&\Gd^-\otimes\Gd^+={\color{black}\wedge^7V^*}\oplus {\color{black}\wedge^5V^*}\oplus {\color{black}\wedge^3V^*}\oplus {\color{black}V^*}\nn\\
\dim=8:&&\Gd^{\pm}\otimes\Gd^{\pm}={\color{blue}\wedge^4_{\pm}V^*}\oplus {\color{red}\wedge^2V^*}\oplus {\color{blue}\bf 1},\label{decomp_8_ev}\\
&&\Gd^-\otimes\Gd^+={\color{black}\wedge^3V^*}\oplus {\color{black}V^*},\nn\\
\dim=10:&&\Gd^-\otimes\Gd^+={\color{black}\wedge^4V^*}\oplus {\color{black}\wedge^2V^*}\oplus {\color{black}\bf 1},\nn\\
&&\Gd^{\pm}\otimes\Gd^{\pm}={\color{blue}\wedge^5_{\mp i}V^*}\oplus {\color{red}\wedge^3V^*}\oplus {\color{blue}\wedge^1V^*},\nn\eea
where the blue summands are symmetric under the exchange of the two factors on the lhs, while the red ones are anti-symmetric.
\begin{remark}
Explicitly, the mappings above are given as follows. Let, say $\dim=16$, $\phi\in\Gd^-$ and $\psi\in \Gd^+$, then
\bea e^{i_1}\wedge \cdots \wedge e^{i_k} \mapsto (\phi,e^{i_1}\cdots e^{i_k}\psi)\label{spin_bi_linear}\eea
gives the map $\phi\otimes\psi\to \wedge^kV^*$ (for $k$ odd).
While for $\dim 10$, since $C$ now involves odd number of $e^i$, thus the same map above would be for $k$ even instead.

We will not go into the proof of the decompositions \eqref{decomp_8_ev}. But the logic goes as: the rhs of \eqref{decomp_8_ev} are all irreducible representations of $\FR{spin}(V)$. Thus if the map \eqref{spin_bi_linear} to a particular summand is nonzero, it will be onto. To show that the map is nonzero, one chooses special $\phi,\psi$ to produce a nonzero image. Having shown that the maps are all onto, then a simple count of dimension shows that the map is isomorphic.
\end{remark}
\smallskip

The rhs of \eqref{spin_bi_linear} are usually called \emph{spinor bilinears}.
The \emph{Fierz identities} are some seemingly magical identities involving the spinor bilinears, responsible for some miraculous cancellations that make certain supersymmetric field theories free of ultra-violet divergence.
But magic does not exist: the Fierz identities are but some manifestations of representation theory of $\FR{spin}(V)$, which we review next.

To start, take $\dim=8$ and let $\psi_{1,2,3,4}\in\Gd^+$. We can decompose the quadruple tensor product $\psi_1\otimes \psi_2\otimes\psi_3\otimes\psi_4$ into tensor representations. In particular, we are interested in the trivial representation in this decomposition. But there are two ways to do this: we decompose $\psi_1\otimes\psi_2$ and $\psi_3\otimes\psi_4$ as in \eqref{decomp_8_ev}
\bea (\Gd^+\otimes\Gd^+)\tensor(\Gd^+\otimes\Gd^+)=(\wedge^4_{\pm}V^*\oplus\wedge^2V^*\oplus{\bf 1})
\tensor(\wedge^4_{\pm}V^*\oplus\wedge^2V^*\oplus{\bf 1})\mapsto {\bf 1}\oplus {\bf 1}\oplus {\bf 1},\nn\\
(\psi_1\otimes\psi_2)\tensor(\psi_3\otimes\psi_4)\mapsto\frac{1}{4!}(\psi_1e^{ijkl}\psi_2)(\psi_3e^{ijkl}\psi_4)
\oplus\frac{1}{2!}(\psi_1e_{ij}\psi_2)(\psi_3e^{ij}\psi_4)\oplus(\psi_1\psi_2)(\psi_3\psi_4)\label{used_fierz_8},\eea
where $e^{i_1\cdots i_p}=e^{[i_1}\cdots e^{i_p]}/p!$.

Alternatively, we decompose $\psi_1\otimes\psi_4$ and $\psi_3\otimes\psi_2$ first, giving
\bea (\psi_1\otimes\psi_4)\tensor(\psi_3\otimes\psi_2)\to \frac{1}{4!}(\psi_1e^{ijkl}\psi_4)(\psi_3e^{ijkl}\psi_2)
\oplus\frac{1}{2!}(\psi_1e^{ij}\psi_4)(\psi_3e^{ij}\psi_2)\oplus(\psi_1\psi_4)(\psi_3\psi_2)\label{used_fierz_8T}.\eea
The tensor decomposition does not depend how one does it, and so the three terms on the rhs must be a linear combination of the three terms in \eqref{used_fierz_8}, e.g.
\bea (\psi_1\psi_4)(\psi_3\psi_2)=\frac{a}{4!}(\psi_1e^{ijkl}\psi_2)(\psi_3e^{ijkl}\psi_4)
+\frac{b}{2!}(\psi_1e^{ij}\psi_2)(\psi_3e^{ij}\psi_4)+c(\psi_1\psi_2)(\psi_3\psi_4)\nn.\eea
Note the coefficients $a,b,c$ do not depend on $\psi_i$, and so we can get these by plugging in specific $\psi_i$'s as follows. Let $\xi^a$ be an orthonormal basis of $\Gd^+$, and $\xi_a$ the dual basis such that $(\xi_b,\xi^b)=\gd_a^b$. We let $\psi_3=\xi_a$, $\psi_4=\xi^a$ and we sum over $a$.
Now since $\xi_a\otimes\xi^a$ is in the trivial representation, \eqref{used_fierz_8} gives
\bea (\psi_1\otimes\psi_2)\tensor(\xi_a\otimes\xi^a)\to0\oplus 0\oplus (\psi_1\psi_2)(\xi_a\xi^a)=8(\psi_1\psi_2)[0,0,1]\nn.\eea
But $(\psi_1\xi^a)(\xi_a\psi_2)=(\psi_1\psi_2)$, giving $c=1/8$. Similarly, set $\psi_3=\xi_a$, $\psi_4=e^{12}\xi^a$, we note that $\xi_a\otimes e^{12}\xi^a$ lives in $\wedge^2V^*$, and so \eqref{used_fierz_8} gives
\bea (\psi_1\otimes\psi_2)\tensor(\xi_a\otimes e^{12}\xi^a)\to 0\oplus \frac{1}{2!}(\psi_1e^{ij}\psi_2)(\xi_ae^{ij}e^{12}\xi^a)\oplus 0=-8(\psi_1e^{12}\psi_2)[0,1,0]\nn.\eea
Here in evaluating $T^{ij,kl}=(\xi_ae^{ij}e^{kl}\xi^a)$, we note that from the action of $\FR{spin}$ in def.\ref{def_spin_action}, it is $\FR{spin}(8)$ invariant. By some basic invariant theory, it must be proportional to $\gd^{ik}\gd^{jl}-\gd^{il}\gd^{jk}$. One can easily get the proportionality constant to be $-8$. On the other hand $(\psi_1e^{12}\xi^a)(\xi_a\psi_2)=(\psi_1e^{12}\psi_2)$, and we get $b=-1/8$.

The coefficient $c$ is the only one with a slight subtlety. We proceed as above, and set $\psi_3=\xi_a$ and $\psi_4=e^{1234}\xi^a$, \eqref{used_fierz_8} gives
\bea (\psi_1\otimes\psi_2)\tensor(\xi_a\otimes e^{1234}\xi^a)\to 0\oplus \frac{1}{4!}(\psi_1e^{ijkl}\psi_2)(\xi_ae^{ijkl}e^{1234}\xi^a)\oplus 0=16(\psi_1e^{1234}\psi_2)[0,1,0]\nn.\eea
But now in evaluating $T^{ijkl;pqrs}=(\xi_ae^{ijkl}e^{pqrs}\xi^a)$, the invariant theory gives us two invariant tensors $\gd^{ip}\gd^{jq}\gd^{kr}\gd^{ls}+\textrm{anti-sym}[pqrs]$ and $\ep^{ijklpqrs}$, with equal coefficient (since the rank 4 tensor is self-dual, see \eqref{decomp_8_ev}). On the other hand $(\psi_1e^{1234}\xi^a)(\xi_a\psi_2)=(\psi_1e^{1234}\psi_2)$, and we get $c=1/16$. Summarising
\bea (\psi_1\psi_2)(\psi_3\psi_4)=\frac18(\psi_1\psi_4)(\psi_3\psi_2)-\frac18\frac{1}{2!}(\psi_1e^{ij}\psi_4)(\psi_3e^{ij}\psi_2)+\frac1{16}\frac{1}{4!}(\psi_1e^{ijkl}\psi_4)(\psi_3e^{ijkl}\psi_2).\nn\eea

Repeating the same argument, we get other similar relations, but as the calculation is hardly instructive, we will just state the result with no proof.
To streamline notations, denote
\bea A_k=\frac{1}{k!}(\psi_1e^{i_1\cdots i_k}\psi_2)(\psi_3e^{i_1\cdots i_k}\psi_4),\nn\\
A^T_k=\frac{1}{k!}(\psi_1e^{i_1\cdots i_k}\psi_4)(\psi_3e^{i_1\cdots i_k}\psi_2)\nn\eea
for $\psi_2,\psi_4$ of the same chirality, while if they are of opposite chirality, we denote the same quantity as $B_k$ or $B_k^T$. We have
\bea D=8~~~\begin{array}{c|c|c|c}
      & A^T_0 & A^T_2 & A^T_4 \\
      \hline
      A_0 & \frac18 & -\frac18 & \frac{1}{16} \\
      \hline
      A_2 & -\frac72 & \frac12 & \frac14 \\
      \hline
      A_4 & \frac{35}{4} & \frac54 & \frac{3}{8}
     \end{array}~~~\begin{array}{c|c|c}
      & A^T_1 & A^T_3 \\
      \hline
      A_1 & -\frac34 & \frac14  \\
      \hline
      A_3 & \frac74 & \frac34 \\
     \end{array}~~~\begin{array}{c|c|c}
      & B^T_1 & B^T_3 \\
      \hline
      B_0 & \frac18 & -\frac18 \\
      \hline
      B_2 & -\frac74 & -\frac14 \\
     \end{array}~~~\begin{array}{c|c|c}
      & B^T_0 & B^T_2 \\
      \hline
      B_1 & 1 & -\frac12 \\
      \hline
      B_3 & -7 & -\frac12 \\
     \end{array}\label{fierz_8}.\eea
Note that in $\dim 8$, $B_4=0$ always.
To read this table, we take $A_4$, then the third row says $A_4=35/4A_0^T+5/4A_2^T+3/8A_4^T$.
It is helpful to check that applying the above relations twice, one does get back to square one.

From these relations, we can derive some handier ones. For example,
\begin{lemma}In 8D, assume that $\xi,\phi,\psi\in \Gd^+$ and $\eta\in\Gd^-$, then
 \bea \frac12 e^{ij}\eta (\phi e^{ij}\psi)=4\phi (\psi\xi)-4\psi (\phi\xi)\label{Fierz_8_I}\\
 2\eta (\phi\psi)=e^i\psi (\phi e^i\eta)+e^i\phi (\psi e^i\eta).\label{Fierz_8_II}\eea
 we remind the reader that $(\phi e^i\psi):=\phi^TC e^i\psi$.
\end{lemma}
\begin{proof}
 For the first relation, we eliminate $A_4^T$ in the first table in \eqref{fierz_8} by using the fact that $(\psi e^{ij}\phi)=-(\phi e^{ij}\psi)$ but $(\psi e^{ijkl}\phi)=(\phi e^{ijkl}\psi)$. For the second equality, one uses the third table and eliminate $B_3^T$ instead.
\end{proof}

We will also need the Fierz relations for $\dim=16$. Keeping the same notations we have
\bea D=16,~~~\begin{array}{c|c|c|c|c|c}
          & A_0^T & A_2^T & A_4^T & A_6^T & A_8^T \\
          \hline
      A_0 & \frac{1}{128} & -\frac{1}{128} & \frac{1}{128} & -\frac{1}{128} & \frac{1}{256} \\
      \hline
      A_2 & -\frac{15}{16} & \frac12 & -\frac{3}{16} & 0 & \frac{1}{32} \\
      \hline
      A_4 & \frac{455}{32} & -\frac{91}{32} & -\frac{9}{32} & \frac{5}{32} & \frac{7}{64} \\
      \hline
      A_6 & -\frac{1001}{16} & 0 & \frac{11}{16} & \frac12 & \frac{7}{32} \\
      \hline
      A_8 & \frac{6435}{64} & \frac{429}{64} & \frac{99}{64} &  \frac{45}{64} & \frac{35}{128}
     \end{array}\nn\eea
One can likewise check that these relations square to 1.
From these relations we can derive a crucial relation
\begin{lemma}
 For $\dim 16$, let $\psi_{1,2,3}\in\Gd^+$, then
 \bea \frac12 e^{ij}\psi_1 (\psi_2e^{ij}\psi_3)+\textrm{cyc perm}(1,2,3)=0.\label{Jacobi_E8}
 \eea
\end{lemma}
\begin{proof}
 One uses the table above to write $A_2$ in terms of $A_{0,2,4,8}$, noting the fact that $(\psi_1e^{ij}\psi_2)=-(\psi_2e^{ij}\psi_1)$, but $(\psi_1e^{i_1\cdots i_k}\psi_2)=(\psi_2e^{i_1\cdots i_k}\psi_1)$ for $k=0,4,8$. It is crucial that $A_6^T$ does not appear here.
\end{proof}

Finally we record the 10D Fierz relations
\bea D=10~~\begin{array}{c|c|c|c}
          & A_0^T & A_2^T & A_4^T \\
          \hline
      A_0 & \frac{1}{16} & -\frac{1}{16} & \frac{1}{16} \\
      \hline
      A_2 & -\frac{45}{16} & \frac{13}{16} & \frac{3}{16} \\
      \hline
      A_4 & \frac{105}{8} & \frac{7}{8} & \frac{1}{8} \\
     \end{array}~~\begin{array}{c|c|c}
          & A_1^T & A_3^T \\
          \hline
      A_1 & -\frac{1}{2} & \frac{1}{4} \\
      \hline
      A_3 & 3 & \frac{1}{2} \\
          \end{array}~~
          \begin{array}{c|c|c|c}
          & B_1^T & B_3^T & B_5^T \\
          \hline
      B_0 & \frac{1}{16} & -\frac{1}{16} & \frac{1}{32} \\
      \hline
      B_2 & -\frac{27}{16} & \frac{3}{16} & \frac{5}{32} \\
      \hline
      B_4 & \frac{21}{8} & \frac{7}{8} & \frac{5}{16} \\
     \end{array}~~\begin{array}{c|c|c|c}
          & B_0^T & B_2^T & B_4^T \\
          \hline
      B_1 & \frac{5}{8} & -\frac{3}{8} & \frac{1}{8} \\
      \hline
      B_3 & -\frac{15}{2} & \frac{1}{2} & \frac{1}{2} \\
      \hline
      B_5 & \frac{63}{4} & \frac{7}{4} & \frac{3}{4} \\
     \end{array}\nn\eea
From these we also derive two important relations
\begin{lemma}
 For $\dim 10$, let $\psi_{1,2,3}\in\Gd^+$ and $\eta\in\Gd^-$, then
 \bea \psi_1 (\eta\psi_2)+\frac12 e^{ij}\psi_1 (\eta e^{ij}\psi_2)-2e^i \eta(\psi_1 e^i\psi_2)=-4\psi_2(\psi_1\eta)\label{Fierz_10_key}\eea
 and the same relation holds if one flips all chiralities above.
 Furthermore
 \bea e^i \psi_1(\psi_2 e^i\psi_3)+\textrm{cyc perm}(1,2,3)=0.\label{Jacobi_E6}\eea
\end{lemma}
\begin{proof}
 For the first relation, one uses the table to rewrite $A_0,A_2,B_1$ and eliminate $A_4^T,B_4^T,A_2^T,B_2^T$. For the second relation, we use the table to rewrite $A_1$ in terms of $A_1^T,A_3^T$. The latter will vanish under cyclic permutation.
\end{proof}

\subsection{$\FR{g}_2$ from Clifford Algebra}\label{sec_g2fCA}
In this section, we illustrate the use of Fierz identity by presenting the octonion multiplication as spinor manipulations.

We have mentioned that in $\dim=8$, both the vector representation $V$, and the two spinor representations $\Gd^{\pm}$ have dimension 8. Furthermore, one can impose reality conditions to $V,\Gd^{\pm}$ and get three dim 8 real vector spaces. All three will be identified as $\BB{O}$. Our construction does not rely on the reality condition, so we leave $V,\Gd^{\pm}$ complex for now and identify them as $\BB{C}\otimes\BB{O}$ instead.

We denote the basis of $\BB{R}^8$ as $e^0,\cdots,e^7$, with $e^0$ intended as the real unit of the octonions.
Pick any
\bea s\in\Gd^+,~~Cs^*=s,~~(ss)=s^TCs=1,\nn\eea
and set $t=e^0s$. From $s,t$ we have maps between $\Gd^{\pm}$ and $V$,
\bea s:~V\to \Gd^-,~~v\to vs,~~~~s^{-1}:~~\Gd^-\to V,~~\eta\to (se^i\eta)e^i,\label{map_s}\\
t:~V\to \Gd^+,~~v\to vt,~~~~t^{-1}:~~\Gd^+\to V,~~\xi\to (te^i\xi)e^i,\label{map_t}\eea
where we identified $V\sim V^*$ using the standard Euclidean inner product. When $V$ is complexfied, we use the same inner-product rather than the hermitian inner product so as to maintain complex linearity.
\begin{lemma}\label{lem_triality_map}
 The maps $s,s^{-1}$ are indeed inverse to each other, the same for $t,t^{-1}$.
\end{lemma}
\begin{proof}
 We apply $s$ then $s^{-1}$
 \bea v\to vs\to (se^ivs)e^i=\frac12(s\{e^i,v\}s)e^i=(ss)v=v,\nn\eea
 where we used $(se^ivs)=(sve^is)$. Now we check $ss^{-1}$
 \bea \eta \to (se^i\eta)e^i\to (se^i\eta)e^is=\eta(ss)=\eta,\nn\eea
 where we used \eqref{Fierz_8_II}.
\end{proof}
\begin{definition}
 We define a multiplication rule on $V$ as
 \bea u\star v:=(t(u)e^is(v))e^i,~~~~u,v\in V.\nn\eea
 This coincides with the octonion multiplication.
\end{definition}
\begin{example}
We compute some examples
\bea &&1\star u=e^0\star u=(te^0e^i us)e^k=(se^ius)e^i=u,\nn\\
&&e^1\star e^1=(te^1e^ie^1s)e^i=-(te^ie^1e^1s)e^i+2(te^1s)e^1=-(se^0e^is)e^i+2\hcancel{(se^0e^1s)}e^1=-e^0,\nn\eea
where we have repeatedly used $(\xi_1uv\xi_2)=(\xi_2vu\xi_1)$ and $\{u,v\}=2u\cdotp v$.
So $e^0$ is the unit while $e^1,\cdots,e^7$ are the imaginary units. We leave it to the reader to show that $\overline{u\star v}=\bar v\star\bar u$, by using $e^0s(u)=t(\bar u)$.

To get more specific multiplication rules, one needs choose a concrete representation of the Clifford algebra and $s$.
 We pick the matrix representation of $e^i$ as in ex.\ref{ex_Gd_and_gamma} (apart from the index shift $12345678\to 01234567$), and
 \bea &&s=\frac{1}{\sqrt2}(|+\!+\!+\!+\ket+|-\!-\!-\!-\ket),\nn\\
 &&C=\gs_1\otimes i\gs_2 \otimes \gs_1\otimes i\gs_2,\nn\\
 &&\gc=\gs_3\otimes\gs_3\otimes\gs_3\otimes\gs_3.\nn\eea
 In this case
 \bea e^2\star e^3&=&(se^0e^2e^ie^3s)e^i=-(se^0e^ie^2e^3s)e^i+2\hcancel{(se^0e^3s)}e^2
 =-\hcancel{(se^2e^3s)}e^0-(se^0e^1e^2e^3s)e^1=e^1,\nn\\
 e^3\star e^7&=&(se^0e^3e^ie^7s)e^i=(se^0e^3e^4e^7s)e^4+\hcancel{(se^0e^3e^5e^7s)}e^5
 =-e^4.\nn\eea
 The detailed table is
   \bea
   \begin{array}{c|cccccccc}
   & e^0 & e^1 & e^2 & e^3 & e^4 & e^5 & e^6 & e^7 \\
   \hline
  e^0 & e^0  &   e^1  &   e^2  &   e^3  &   e^4  &  e^5  &  e^6  & e^7 \\
  e^1 & e^1  & -e^0  &   e^3  &  -e^2  &   e^5  &  -e^4  &   e^7  &  -e^6 \\
  e^2 &  e^2  &  -e^3  & -e^0  &   e^1  &  -e^6  &   e^7  &  e^4  &  -e^5 \\
  e^3 &  e^3  &   e^2  &  -e^1  & -e^0  &   e^7  &   e^6  &  -e^5 &   -e^4 \\
  e^4 &  e^4  &  -e^5  &  e^6  &  -e^7 &  -e^0  &  e^1 &   -e^2  &  e^3 \\
  e^5 &  e^5  & e^4  &  -e^7  &  -e^6 &   -e^1 &  -e^0 &  e^3  & e^2 \\
  e^6 &  e^6  &  -e^7  &  -e^4  &   e^5  &  e^2 &   -e^3 &  -e^0  & e^1 \\
  e^7 &  e^7  &   e^6  &   e^5  &   e^4  &  -e^3 &   -e^2 &   -e^1 &  -e^0 \\
   \end{array}\label{octonion_tbl}\eea
This table does not necessarily agree with the standard multiplication rule, but we will prove that all the known octonion identities are fulfilled.
\end{example}

\begin{proposition}\label{prop_alternating}
  The multiplication satisfies
  \bea u\star \bar u= u\cdotp u, ~~~ u\star v+v\star u=2(u_0v+v_0u-u\cdotp v),
  ~~~ u\star \bar v+v\star  \bar u=2u\cdotp v,\label{oct_prod_sym}\eea
  where $\cdotp$ is the Euclidean inner-product (complexified) and $\bar u=u_0e^0-\sum_{i=1}^7u_ie^i$.

  The multiplication is not associative, but the associator
  \bea [u,v,w]=u\star (v\star w)-(u\star v)\star w\nn\eea
  is totally anti-symmetric. As a direct result, we have
  \bea& (u\star v)\star u=u\star (v\star u)=-(u\cdotp u)\bar v+2(u\cdotp\bar v)u,\nn\\
    &u\star (\bar u\star v)=(u\cdotp u)v.\nn\eea
\end{proposition}
\begin{proof}
We compute the first (the $i,j$ indices go from 0 to 7 while $a,b$ from 1 to 7)
\bea u\star \bar u&=&(t ue^i\bar us)e^i=
 u_0^2(te^0e^ie^0s)e^i+u_0u_a(te^ae^ie^0-e^0e^ie^a)e^i-\frac12u_au_b(t(e^ae^ie^b+e^be^ie^a))e^i\nn\\
 &=&u_0^2+u_0u_a(s(e^0e^ae^ie^0-e^ie^a)s)e^i-\frac12u_au_b(t(-e^ae^be^i+2e^a\gd^{ib}-e^be^ae^i+2e^b\gd^{ia})e^i\nn\\
 &=&u_0^2+2u_0u_a\hcancel{(se^0e^as)}e^0-\frac12u_au_b(se^0(-2\gd^{ab}e^i+4\hcancel{e^a\gd^{ib}})s)e^i=u_0^2+u_a^2.\nn\eea
Applying this relation to $u+v$ and $u+\bar v$, we get the other two relations.

Let us compute the associator
\bea [u,v,w]_i=(tue^ie^js)(tve^jws)-(tue^jvs)(te^je^iws).\nn\eea
We apply \eqref{Fierz_8_II} to both terms
\bea (tue^ie^js)(tve^jws)&=&-(tue^ie^js)(tvwe^js)+2(tue^iws)(tvs)\nn\\
&=&-(tue^ie^js)(se^jwvt)+2(tue^iws)(tvs)=-(tue^iwvt)+2v_0(u\star w)_i\nn\\
(tue^jvs)(te^je^iws)&=&-(te^juvs)(te^je^iws)+2(tvs)(tue^iws)\nn\\
&=&-(svue^iws)+2v_0(u\star w)_i\nn\\
\left[u,v,w\right]_i&=&-(tue^iwvt)+(svue^iws).\nn\eea
If we want to symmetrise $u,v$, we can write
\bea [u,v,w]_i&=&-(te^iwuvt)-2u_i(w\cdotp v)+2v_i(u\cdotp w)+(svue^iws),\nn\\
(u \leftrightarrow v)&\To& -2w_i(u\cdotp v)+2(v\cdotp u)w_i=0.\nn\eea
In the same way the associator is anti-symmetric in $v,w$ and also in $w,u$.

For the final two identities, clearly $(u\star v)\star u$ and $u\star (v\star u)$ are equal due to the anti-symmetry of the associator. We have computed above that
\bea u\star (v\star u)&=&-(tue^iuvt)e^i+2v_0(u\star u)=(u\cdotp u)v-2(u\cdotp v)u+2v_0(2u_0u-u\cdotp u)\nn\\
&=&-(u\cdotp u)\bar v+2(u\cdotp \bar v)u.\nn\eea
The last identity follows from $[u,\bar u,v]=[u,-u+2u_0,v]=-[u,u,v]=0$, and so $u\star(\bar u\star v)=(u\star \bar u)\star v=(u\cdotp u)v$.
\end{proof}
\begin{definition}
 Let $V=\BB{O}$ be the vector representation of $Spin(8)$, the $G_2$ group is the unique simply connected subgroup $G_2\subset Spin(7)\subset Spin(8)$ acting on the imaginary octonions that preserves their multiplication rule.
\end{definition}
The multiplication rule here is defined after the choice of one real spinor $s$. Thus the subgroup in question needs only preserve $s$.
From our explicit choice of $s$, we can easily get the Lie algebra $\FR{g}_2$. First, note that the sub-algebra $\FR{su}(3)\subset \FR{spin}(6)$ annihilates $s$, giving us 8 generators of $\FR{g}_2$. Apart from these, we have
\bea \frac14(e^2+ie^3)(e^4+ie^5)-\frac{i}2 e^1(e^6-ie^7),\nn\\
\frac14(e^4+ie^5)(e^6+ie^7)-\frac{i}2 e^1(e^2-ie^3),\nn\\
\frac14(e^6+ie^7)(e^2+ie^3)-\frac{i}2 e^1(e^4-ie^5),\label{g2/su3}\eea
plus their complex conjugates. Altogether we get the 14 dimensional Lie algebra $\FR{g}_2$.

\smallskip

We have identified $V$ as $\BB{O}$, but we could have also done the same for $\Gd^{\pm}$ by using the $s,t$ maps in lem.\ref{lem_triality_map} to first identify $\Gd^{\pm}$ with $V$ and then with $\BB{O}$. Thus we can also present the octonion multiplication as a map of e.g. $V\otimes \Gd^+\to \Gd^-$. It turns out that this map is none other than the Clifford action. We first define the bar action
\begin{definition}\label{def_bar}
  Let $x\in V$, $\xi\in \Gd^+$ and $\eta\in\Gd^-$ (not assumed real), let
  \bea \bar x=2(x\cdotp e^0)e^0-x,~~~~\bar\xi=2(s\xi)s-\xi,~~~~\bar\eta=2(t\eta)t-\eta.\nn\eea
  In particular, these actions are $\BB{C}$-linear.
\end{definition}
\begin{proposition}\label{prop_oct_mult_acatar}
  Let $x,\xi,\eta$ be as above,
  \bea t(x\star s^{-1}(\eta))=x\bar\eta, && t(s^{-1}(\eta)\star x)=\overline{\bar x\eta}\nn\\
  s(t^{-1}(\xi)\star x)=x\bar\xi,&&s(x\star t^{-1}(\xi))=\overline{\bar x\xi}\nn\\
  t^{-1}(\xi)\star s^{-1}(\eta)=(\xi e^i\eta)e^i,&&s^{-1}(\eta)\star t^{-1}(\xi)=\overline{(\bar\eta e^i\bar\xi)e^i},\nn\\
  t^{-1}(\bar\xi_1)\star t^{-1}(\xi_2)+(1\leftrightarrow2)=2(\xi_1\xi_2),&&s^{-1}(\bar\eta_1)\star s^{-1}(\eta_2)+(1\leftrightarrow2)=2(\eta_1\eta_2).\nn\eea
\end{proposition}
These equalities show precisely how the octonion multiplications are related to the Clifford actions.
\begin{proof}
We will demonstrate a few items to give an idea of the proof.
\bea &&x\star s^{-1}(\eta)=(txe^ie^js)(se^j\eta)e^i=(txe^i\eta)e^i,\nn\\
&&t(x\star s^{-1}(\eta))=e^it(txe^i\eta)=-e^it(te^ix\eta)+2xt(t\eta)=-x\eta+2xt(t\eta)=x\bar\eta\nn\eea
where we have used \eqref{Fierz_8_II} repeatedly.

Let us also compute $s^{-1}(\eta)\star t^{-1}(\xi)$
\bea s^{-1}(\eta)\star t^{-1}(\xi)&=&(se^i\eta)(te^ie^ke^js)(te^j\xi)e^k=(se^i\eta)(se^0e^ie^ke^js)(te^j\xi)e^k\nn\\
&=&-(se^i\eta)(se^ie^ke^jt)(te^j\xi)e^k+2(t\eta)(se^ke^js)(te^j\xi)e^k\nn\\
&&-2(se^i\eta)(se^ie^js)(te^j\xi)e^0+2(se^i\eta)(se^ie^ks)(te^0\xi)e^k\nn\\
&=&-(\eta e^k\xi)e^k+2(t\eta)(te^j\xi)e^j-2(se^i\eta)(te^i\xi)e^0+2(se^i\eta)(s\xi)e^i\nn\\
&=&-(\bar\eta e^k\bar \xi)e^k+4(t\eta)(s\xi)e^0-2(se^i\eta)(te^i\xi)e^0\nn\\
&=&-(\bar\eta e^k\bar \xi)e^k+4(t\eta)(s\xi)e^0+2(se^i\eta)(se^ie^0\xi)e^0-4(t\eta)(s\xi)e^0\nn\\
&=&-(\bar\eta e^k\bar \xi)e^k+2(\eta e^0\xi)e^0=\overline{(\bar\eta e^k\bar \xi)e^k}\nn.\eea
And we compute the final relation
\bea t^{-1}(\bar\xi_1)\star t^{-1}(\xi_2)&=&(te^i\bar\xi_1)(te^ie^ke^js)(te^j\xi_2)e^k=(\bar\xi_1e^ke^js)(te^j\xi_2)e^k\nn\\
&=&-(\bar\xi_1e^ke^0\xi_2)e^k+2(\bar\xi_1e^ke^0s)(s\xi_2)e^k\nn\\
&=&(\xi_1e^ke^0\xi_2)e^k-2(\xi_1s)(se^ke^0\xi_2)e^k-2(\xi_1e^ke^0s)(s\xi_2)e^k+4(\xi_1s)(s\xi_2)e^0\nn\\
&=&(\xi_1e^ke^0\xi_2)e^k-2(\xi_1s)(se^ke^0\xi_2)e^k+2(se^ke^0\xi_1)(s\xi_2)e^k.\nn\eea
After symmetrising $1,2$, only the first term survives: $(\xi_{(1}e^ke^0\xi_{2)})e^k=2(\xi_1\xi_2)e^0$.
\end{proof}

\section{Adam's Construction of exceptional Lie algebra}
The exceptional Lie algebra can be constructed using the octonions, see \cite{Tits} \cite{FREUDENTHAL1964145} (also \cite{Barton:2000ki} for a variant construction, while \cite{Ramond_exp} contains an exposition more to the taste of physicists). But in \cite{Adams1996lectures} a very different, but more streamlined approach is used. The basic idea is to start from a Lie algebra $\FR{g}$, and a representation $V$ of $\FR{g}$.
One can define the Lie bracket between $\FR{g}$ and $V$ by the Lie algebra action: $[x,v]:=x\circ v$ for $x\in\FR{g}$ and $v\in V$. If one can find a bi-linear map $V\otimes V\to \FR{g}$, then one can define a Lie bracket $[V,V]\to\FR{g}$. One needs to check two things, first $[V,V]\to\FR{g}$ may not be anti-symmetric. But this is not a big issue, in case this bracket is symmetric, one gets potentially a super-Lie-algebra. A more stringent condition is the Jacobi identity i.e. whether or not $[V,[V,V]]=0=[\FR{g},[V,V]]$ under cyclic permutation. It turns out that there are only finitely many $\FR{g}$ and $V$ that can fulfil these criteria.

\subsection{$E_8$}
We start with the Lie algebra $\FR{e}_8$. As a vector space, we take
\bea \FR{e}_8\simeq \FR{so}(16)\oplus \Gd^+\nn\eea
of total dimension $120+128=248$.
\begin{definition}\label{def_e8}
 The bi-linear maps
 \bea &&\FR{g}\otimes \Gd^+\to \Gd^+:~[x,\xi]=x\circ\xi,\nn\\
 &&\Gd^+\otimes\Gd^+\to \FR{g}:~[\xi,\eta]:=-\frac{1}{4}(\xi e^{ij}\eta) x^{ij},~~x^{ij}=e^i\wedge e^j\nn\eea
 along with the Lie bracket of $\FR{so}(16)$ give the Lie algebra structure of $\FR{e}_8$.

 The smallest representation of $E_8$ is its adjoint representation of dimension 248.
\end{definition}
\begin{proof}
 First $[\xi,\eta]$ as defined above is clearly anti-symmetric. We need only check Jacobi-identity. Examine first
 \bea [x,[y,\xi]]+[y,[\xi,x]]+[\xi,[x,y]]\stackrel{?}{=}0.\nn\eea
 That this is zero just says that $\Gd^+$ is a representation of $\FR{g}$.
 Next we check the Jacobi identity involving $x,\xi,\eta$, we will need def.\ref{def_spin_action} for the action of $\FR{so}$ on $\Gd$ (we take $x=e^p\wedge e^q$)
 \bea &&[x,[\eta,\xi]]+[\eta,[\xi,x]]+[\xi,[x,\eta]]
 =[\frac12e^{pq},-\frac{1}{8}(\eta e^{rs}\xi)e^{rs}]+\frac{1}{8}(\eta e^{rs}\frac12e^{pq}\xi)e^{rs}-\frac{1}{8}(\xi e^{rs}\frac12e^{pq}\eta)e^{rs}\nn\\
 &=&-\frac{1}{16}(\eta e^{rs}\xi)[e^{pq},e^{rs}]+\frac{1}{16}(\eta[e^{rs},e^{pq}]\xi)e^{rs}\nn\\
 &=&\frac{1}{4}(\eta e^{s[q}\xi)e^{p]s}+\frac{1}{4}(\eta e^{s[p}\xi)e^{q]s}=0.\nn\eea
 The most non-trivial Jacobi is between three spinors
 \bea [\xi,[\eta,\zeta]]+\textrm{cyc perm}(\xi,\eta,\zeta)=\frac{1}{32}(\eta e^{pq}\zeta)e^{pq}\xi+\textrm{cyc perm}(\xi,\eta,\zeta).\nn\eea
 One realises that the rhs is just \eqref{Jacobi_E6}.
\end{proof}
\begin{definition}(Ch.7 \cite{Adams1996lectures})
Pick a definite Killing form of the real vector space $\FR{e}_8$, then the automorphism $\opn{Aut}(\FR{e}_8)$ is a compact subgroup of $SO(\FR{e}_8)$. Its identity component is the compact real Lie group $E_8$.
\end{definition}

\subsection{$E_6$}\label{sec_E6}
We skip over $E_7$ and go directly to $E_6$. The construction is very concisely given in the beginning few paragraphs of chapter 8 of \cite{Adams1996lectures}.

We split $V=\BB{R}^{16}$ into $\BB{R}^{10}\oplus\BB{R}^6$. We likewise break $\FR{so}(16)$ to $\FR{so}(10)\oplus \FR{so}(6)$.
There is a well-known embedding $\FR{u}(3)\hookrightarrow \FR{so}(6)$ by regarding $\BB{R}^6$ as $\BB{C}^3$. We further split $\FR{u}(3)=\FR{su}(3)\oplus \FR{u}(1)$ with the last $\FR{u}(1)$ acting on all three factors of $\BB{C}^3$ with identical weight 1, we name it $\varrho$.
\begin{definition}\label{def_e6}
The $\FR{e}_6$ Lie algebra is the sub-algebra of $\FR{e}_8$ that centralises $\FR{su}(3)$. The $E_6$ group is the centraliser of $SU(3)\subset E_8$. And $E_6$ contains a subgroup of the same rank
\bea Spin(10)\times_{\BB{Z}_4}U(1).\nn\eea
\end{definition}
Let us unpack this definition a bit. Let
\bea \varrho=2e^{11}\wedge e^{12}+2e^{13}\wedge e^{14}+2e^{15}\wedge e^{16},\label{varrho}\eea
be the $\FR{u}(1)$ above (in particular $\varrho$ acts on spinors as $e^{11,12}+e^{13,14}+e^{15,16}$, see def.\ref{def_spin_action}). The $\FR{su}(3)$ part is generated by
\bea (e^{a}+ie^{a+1})\wedge (e^{b}-ie^{b+1}),~~~a,b=11,13,15,\nn\eea
Clearly the entire $\FR{so}(10)$ and $\varrho$ centralise $\FR{su}(3)$. As for $\Gd^+$, we need to describe the spinors that are annihilated by the generators of $\FR{su}(3)$ above.
Looking at our explicit choice of $\Gd$ in \ref{ex_Gd_and_gamma} and how the Clifford acts on $\Gd$, we see that the spinors below are annihilated by $\FR{su}(3)$
\bea \Gd_{10}\otimes |+\!++\ket,~~~\Gd_{10}\otimes |-\!--\ket\nn\eea
where we use $|\pm\ket$ to denote the basis of $\BB{C}^2$. Furthermore, the 16D chirality operator is
\bea \underbrace{\gs_3\otimes \cdots \otimes \gs_3}_{8}=\gc_{10}\otimes \gs_3\otimes\gs_3\otimes\gs_3,\nn\eea
hence $\Gd^+_{16}$ must have even number of $-$'s. Thus we are left with two kinds of spinors
\bea \psi\otimes |+\!++\ket,~~\phi\otimes|-\!--\ket,~~~\psi\in\Gd^+_{10},~~\phi\in\Gd^-_{10}\nn\eea
It is also clear that $\varrho$ would act on these spinors with weight $\pm3$.
From now on, we will no longer write the $|+\!++\ket,|-\!--\ket$ factors, rather we just keep in mind that $\Gd^{\pm}_{10}$ has $\varrho$ weight $\pm3$ (actually $\pm 3i$, but it is customary to omit the $i$). The Lie bracket of $\FR{e}_6$ can be read off from that of $\FR{e}_8$, we collect these into
\begin{definition}\label{def_e6_bracket}
The $\FR{e}_6$ Lie algebra decomposes as $\FR{so}(10)\oplus \FR{u}(1)\oplus \Gd_{10}^+\oplus\Gd^-_{10}$. Let $x\in\FR{so}(10)$, $\varrho\in\FR{u}(1)$, $\xi\in\Gd^+$ and $\eta\in \Gc^-$, we have the Lie brackets
\bea &&[x,\xi]=x\circ\xi,~~[x,\eta]=x\circ\eta,\nn\\
&&\left[\xi,\eta\right]=-\frac{1}{4}(\xi e^{ij}\eta)x^{ij}+\frac{i}{4}(\xi\eta)\varrho\nn\\
&&\left[\varrho,\xi\right]=3i\xi,~~[\varrho,\eta]=-3i\eta.\nn\eea
We remind the reader that $x\circ \xi$ means the Clifford action on $\xi$, i.e.  $x^{ij}\circ\xi=(1/2)e^{ij}\xi$. This Lie algebra has a Hermitian Killing form
\bea \bra x+\phi+r\varrho,y+\psi+s\varrho\ket=-\frac{1}{8}\Tr[xy]+\phi^{\dag}\psi+12rs\nn\eea
\end{definition}
\begin{proof}
That the centraliser should form a Lie algebra is not surprising. Equally straightforward is the first line of the Lie bracket. As for the second line, let us recall that $\xi$, resp. $\eta$ is secretly accompanied by $|+\!++\ket$ resp. $|-\!--\ket$. Recall the $\FR{e}_8$ Lie bracket between two spinors in def.\ref{def_e8} i.e.
\bea [\psi,\phi]=-\frac14(\psi^TC_{16}e^{\mu\nu}\phi)x^{\mu\nu},~~~\psi=\xi\otimes|+\!++\ket,~~\phi=\eta\otimes|-\!--\ket \nn\eea
where we set $\mu\nu\in\{1,\cdots,16\}$, $ij\in\{1,\cdots 10\}$ and $p,q\in\{11,\cdots,16\}$ temporarily. The $C_{16}$ splits into $C_{10}\otimes C$ with $C=i\gs_2\otimes\gs_1 \otimes i\gs_2$. There are only two possibilities when the Lie bracket is not zero
\bea [\psi,\phi]&=&-\frac14(\xi^TC_{10}e^{ij}\eta)(\bra+\!++|C|-\!--\ket)x^{ij}-\frac14(\xi^TC_{10}\eta)(\bra+\!++|Ce^{pq}|-\!--\ket)x^{pq}\nn\\
&=&-\frac14(\xi e^{ij}\eta)x^{ij}-2\times\frac14(\xi\eta)(-i)(x^{11,12}+x^{13,14}+x^{15,16})\nn\\
&=&-\frac14(\xi e^{ij}\eta)x^{ij}+\frac{i}{4}(\xi\eta)\varrho.\nn\eea
\end{proof}
Finally, for the subgroup of maximal rank stated in def.\ref{def_e6}. We observe the $SU(3)$ subgroup sits in $E_8$ as
\bea SU(3)\subset Spin(6)\subset Spin(16)\to E_8.\nn\eea
Even though the last map is not injective, its kernel is a $\BB{Z}/2\BB{Z}$ which intersects $SU(3)$ trivially, and so $SU(3)$ is a subgroup.
It turns out that $Spin(10)$ and $U(1)$ are the obvious and only centraliser of $SU(3)$, but the map $Spin(10)\times U(1)\to E_6$ has a kernel $\BB{Z}/4\BB{Z}$.
To see this, we note that $g=e^1\cdots e^{10}$ is, according def.\ref{def_Spin} and the remark afterwards, an element of $Spin(10)$. We see easily that $g^2=-1$, in fact $g=i\gc$ according to \eqref{chirality}. The action of $g$ on $\Gd^{\pm}$ is by a factor $\pm i$, which is the same as $-i\in U(1)$ if we remember that $U(1)$ acts on $\Gd^{\pm}$ with weight $\pm 3$.

\subsection{Representations of $E_6$}
We can decompose the adjoint representation of $E_8$ into representations of $E_6$, amongst these the 27 and its conjugate will be crucial for us.
\begin{definition}\label{def_27_e6}(Cor 11.5 \cite{Adams1996lectures})
 The fundamental representation of $\FR{e}_6$ has 27 dimension, as a representation of $\FR{so}(10)\oplus\FR{u}(1)$ is decomposes as
 \bea {\bf 27}=V_2\oplus \Gd^+_{-1}\oplus 1_{-4}\nn\eea
 where $V$ is the vector representation of $\FR{so}(10)$ and the subscript denotes the weight under $\FR{u}(1)$. The matrix representing $(x,\xi,\eta,\varrho)\in\FR{e}_6$ can be cast in block form as
 \bea &&x\circ\begin{bmatrix}
        v \\ \psi \\ s
      \end{bmatrix}=\begin{bmatrix}
        xv \\ x\circ \psi \\ 0
      \end{bmatrix},~~~\varrho\circ\begin{bmatrix}
        v \\ \psi \\ s
      \end{bmatrix}=\begin{bmatrix}
        2iv \\ -i\psi \\ -4is
      \end{bmatrix},\nn\\
      &&\xi\circ\begin{bmatrix}
        v \\ \psi \\ s
      \end{bmatrix}=\begin{bmatrix}
        \frac{1}{\sqrt2}(\xi e^i\psi) e^i \\ s\xi \\ 0
      \end{bmatrix},~~\eta\circ\begin{bmatrix}
        v \\ \psi \\ s
      \end{bmatrix}=\begin{bmatrix}
       0 \\ -\frac{1}{\sqrt2} v\eta \\ -(\eta\psi)
      \end{bmatrix}.\nn\eea
%The ${\bf 27}^*$ replaces the right hander $\psi$ with a left hander
\end{definition}
\begin{proposition}\label{prop_d_tensor}
 The dim 27 representation has a cubic invariant $d({\textrm -},{\textrm -},{\textrm -})$. Let $\Psi_{1,2,3}\in{\bf 27}$ with $\Psi_i=[v_i,\psi_i;s_i]$, set
 \bea d(\Psi_1,\Psi_2,\Psi_3)=(v_1\cdotp v_3)s_2-\frac{1}{\sqrt2}(\psi_1v_3\psi_2)+\textrm{cyc perm}.\nn\eea
 It is totally symmetric in $\Psi_{1,2,3}$.

 We can dualise $\Psi_3$ and regard $d$ as a map $\diamond:~{\bf 27}\otimes^s{\bf 27}\to {\bf 27}^*$, it reads
 \bea \Psi_1\diamond\Psi_2=\begin{bmatrix}
                             v_1s_2+v_2s_1-\frac1{\sqrt2}(\psi_1 e^i\psi_2)e^i \\
                             -\frac{1}{\sqrt2}(v_1\psi_2+v_2\psi_1) \\
                             v_1\cdotp v_2
                            \end{bmatrix}.\nn\eea
\end{proposition}
\begin{proof}
 The $d$ tensor is manifestly invariant under $Spin(10)\times U(1)$, we demonstrate the proof that $d$ is also invariant under the action of $\xi$ or $\eta$.
 \bea d(\xi\circ\Psi_1,\Psi_2,\Psi_3)=
 \frac{1}{\sqrt2}((\xi v_3\psi_1)s_2+(\xi v_2\psi_1)s_3)-\frac{s_1}{\sqrt2}((\xi v_3\psi_2)+(\xi v_2\psi_3))-\frac{1}{\sqrt2}(\psi_2 e^i \psi_3)\frac{1}{\sqrt2}(\xi e^i\psi_1).
 \nn\eea
 Note the last term vanishes after taking cyclic permutation, thanks to \eqref{Jacobi_E6}, while the remaining terms cancel trivially.
\end{proof}
\begin{corollary}\label{cor_plucker}
 Let $\Psi$ be in the $E_6$-orbit of the lowest weight vector i.e. of $\Psi_0=[0;0;1]$, then
 \bea d(\Psi,\Psi,-)=0.\nn\eea
 The same applies to the highest weight of ${\bf 27}^*$.
\end{corollary}
\begin{proof}
Since $d$ is an $E_6$ invariant tensor $d(g\Psi_0,g\Psi_0,-)=d(\Psi_0,\Psi_0,g^{-1}-)$. But by looking at the explicit form of the $d$ tensor $d(\Psi_0,\Psi_0,-)=0$. In fact, even more directly, two factors of $\Psi_0$ would have $U(1)$ weight $-4\times2$, yet the vectors in ${\bf 27}^*$ have weights $-2,1,4$, we get the same conclusion.
\end{proof}
Finally, we observe the $U(1)$ subgroup of $E_6$ acts on $\Psi_0$ by a phase, while $Spin(10)$ preserves it. In fact, we have
\begin{lemma}\label{lem_stab_grp_e6}
  The stability group of $\Psi_0\in{\bf 27}$ is precisely $Spin(10)$.
\end{lemma}
\begin{proof}
  In block matrix form, we can write $g\in\opn{Stab}(\Psi_0)$ as
  \bea g=
  \begin{bmatrix}
    g_{11} & g_{12} & 0 \\ g_{21} & g_{22} & 0 \\ \cancel{g_{31}} & \cancel{g_{32}} & 1 \\
  \end{bmatrix},\nn\eea
  where $g_{11}\in\hom(V,V)$, $g_{12}\in\hom(\Gd^+,V)$ etc, the notation should be self-explanatory. Since $E_6$ is a subgroup of $SU(27)$, the conjugate $g^{\dag}=g^{-1}$ also preserves $\Psi_0$, giving $g_{31}=g_{32}=0$.

  Conjugating the Lie algebra of $E_6$ by $g$ should stay in $\FR{e}_6$. So we consider $g\xi g^{-1}\in\FR{e}_6$, in particular its 13-component should vanish. Referring to the matrix corresponding to $\xi$ in def.\ref{def_27_e6}, we get $g_{12}=0$. In the same way, $g_{21}=0$. Thus $g$ is block-diagonal.

  Now notice that in the pairing ${\bf 27}^*\otimes {\bf 27}\to\BB{C}$, the $V$ component is just the (complexified) Euclidean inner-product on $V$, and $g_{11}$ must preserve it, leading to $g_{11}\in O(V)$.

  We know that $g_{22}$ is an invertible map $\hom(\Gd^+,\Gd^+)$.
  We consider again the action of $g\xi g^{-1}$ on the various components
  \bea g\xi g^{-1}\begin{bmatrix}
    0 \\ \psi \\ 0 \end{bmatrix}=
    \begin{bmatrix}
      \frac{1}{\sqrt2}(\xi e^i g_{22}^{-1}\psi) g_{11}e^i \\ 0 \\ 0
 \end{bmatrix},~~~(g\xi g^{-1})\begin{bmatrix}
    0 \\ 0 \\ 1 \end{bmatrix}=\begin{bmatrix}
    0 \\ g_{22}\xi \\ 0 \end{bmatrix}.\nn\eea
If $g\xi g^{-1}$ were to belong to $\FR{e}_6$, then we have
\bea (\xi e^i g_{22}^{-1}\psi) g_{11}e^i=(\xi g_{22}^te^i\psi) e^i~~\To~~
g_{22}^te^kg_{22}^{-1}=g_{11}e^k,\nn\eea
where $t$ is the automorphism \eqref{t_ope}, and we have also used \eqref{t_T}: $C^{-1}g_{22}^TC=g_{22}^t$. This says the twisted adjoint action of $g_{22}$ covers the action of $g_{11}\in O(V)$ (c.f. \eqref{twisted_adj} and remember $\ga$ acts trivially on $g_{22}\in Cl_0$), i.e. $g_{22}\in\Gc(V)$. Next we establish that $N(g_{22})=1$. We use the fact that $g_{22}\in U(\Gd^+)$, which gives $1=g_{22}g_{22}^{\dag}=g_{22}g_{22}^t=N(g_{22})$. We get therefore $g_{22}\in Pin(V)$. But as $g_{22}$ is even, we get finally that $g_{22}\in Spin(V)$ covering $g_{11}\in SO(V)$.
\end{proof}

Before leaving this section, we record a useful lemma for the $\diamond$-product
\begin{lemma}\label{lem_diamond_E6}
  The $\diamond$-product satisfies some Jacobi-like identities: for $\Psi_i\in{\bf 27}$
  \bea (\Psi_1\diamond\Psi_2)\diamond(\Psi_3\diamond\Psi_4)-\Psi_3d(\Psi_1,\Psi_2,\Psi_4)+cyc(1,2,3)=\Psi_4 d(\Psi_1,\Psi_2,\Psi_3).\nn\eea
  For $\Psi_i\in{\bf 27},~\Phi\in{\bf 27}^*$
  \bea \Psi_1\diamond((\Psi_2\diamond\Psi_3)\diamond\Phi)-\bra\Phi,\Psi_3\ket \Psi_1\diamond\Psi_2+cyc(1,2,3)=\Phi d(\Psi_1,\Psi_2,\Psi_3).\nn\eea
\end{lemma}
\begin{proof}
  The proof uses the formula in prop.\ref{prop_d_tensor} and the Fierz identity \eqref{Jacobi_E6}. The calculation is lengthy but unilluminating and so not inflicted upon the reader.
\end{proof}

\subsection{$F_4$}\label{sec_F4}
We follow chapter 8 of \cite{Adams1996lectures}.
In defining $\FR{g}_2$ and $\FR{e}_8$, we used the Clifford algebra of $\BB{R}^8$ and $\BB{R}^{16}\simeq \BB{R}^8\oplus\BB{R}^8$ respectively. Now we identify the $\BB{R}^8$ used for $\FR{g}_2$ as the second $\BB{R}^8$ factor in $\BB{R}^{16}$. Thus, there is a natural embedding $\FR{g}_2\subset\FR{e}_8$.
\begin{definition}
 The $\FR{f}_4$ algebra is the sub-Lie-algebra of $\FR{e}_8$ that centralises $\FR{g}_2$. It is in fact a subalgebra of $\FR{e}_6$.
\end{definition}
\begin{proof}
 Since $\FR{g}_2$ contains the subalgebra $\FR{su}(3)$ that acts on the last six coordinates of $\BB{R}^{16}$. Thus the centraliser of $\FR{g}_2$ must centralise $\FR{su}(3)$, and so must land in $\FR{e}_6$, based on def.\ref{def_e6}.
\end{proof}
Note that so far the description of $\FR{f}_4$ does not depend on how one picks the representation of the Clifford algebra. Next, we want to make the construction more explicit, so we fix our choice of Clifford representation as in Ex.\ref{ex_Gd_and_gamma}.
\begin{proposition}
  In terms of the Lie algebra of $\FR{e}_6$: $(x,\xi,\eta,\varrho)\in \FR{so}(10)\oplus \Gd^+\oplus\Gd^-\oplus \FR{u}(1)$, the $\FR{f}_4$ Lie algebra is described as
  \bea \FR{f}_4=\{(x,\xi,\eta,0)\in\FR{e}_6|x\in\FR{so}(9),~e^{10}\xi=i\eta\}\nn\eea
  and is of dimension $36+16=52$.
\end{proposition}
\begin{proof}
 The generators of $\FR{g}_2$ not in $\FR{su}(3)$ are listed in \eqref{g2/su3}. One sees
 immediately that $\FR{so}(9)\subset\FR{so}(10)$ (acting on the first nine coordinates) centralises $\FR{g}_2$, but $\FR{u}(1)$ does not. As for the $\xi,\eta$ generators in def.\ref{def_e6_bracket}, they will not centralise $\FR{g}_2$ except when they satisfy
 \bea e^{10}\xi=i\eta,~~e^{10}\eta=-i\xi.\nn\eea
 To see this, we keep in mind that $\xi$ has a factor of $|+\!++\ket$ hanging implicitly, while $\eta$ has $|-\!--\ket$. Let us pick one generator from \eqref{g2/su3} and check if it kills $\xi+\eta$, e.g.
 \bea 0\stackrel{?}{=}(\frac14(e^{15}+ie^{16})(e^{11}+ie^{12})-\frac{i}2 e^{10}(e^{13}-ie^{14}))(\xi\otimes|+\!++\ket+\eta\otimes|-\!--\ket),\nn\eea
 where we have renamed some indices in keeping with our embedding of the $\BB{R}^8$ factor.
 Now we need to use the explicit choice of the representation for $e^i$ to do the calculation. We have, after some direct calculation
 \bea &&\frac14(e^{15}+ie^{16})(e^{11}+ie^{12})(\xi\otimes|+\!++\ket+\eta\otimes|-\!--\ket)
 =-\eta\otimes|+\!-+\ket,\nn\\
 &&-\frac{i}2 e^{10}(e^{13}-ie^{14})(\xi\otimes|+\!++\ket+\eta\otimes|-\!--\ket)
 =-ie^{10}\xi\otimes |+\!-+\ket.\nn\eea
 For the sum to cancel, we need $\eta=-ie_{10}\xi$. We need to do the same check for all generators of \eqref{g2/su3} and their complex conjugates.
\end{proof}

Now we can come to the representation. Consider the vector $\Psi_{\emptyset}\in{\bf 27}$ of $\FR{e}_6$
\bea \Psi_{\emptyset}=\frac{\sqrt2}{\sqrt3}\begin{bmatrix}
                                   ie_{10} \\ 0 \\ \frac{1}{\sqrt2}
                                  \end{bmatrix}\label{Psi_0}\eea
it is clearly killed by $\FR{so}(9)$. By using the explicit action of $\FR{e}_6$ in ${\bf 27}$ in def.\ref{def_27_e6}, we can check that $\Psi_{\emptyset}$ is also annihilated by $\xi-ie^{10}\xi$, i.e. $\Psi_{\emptyset}$ is invariant under $F_4$. So we have the split ${\bf 27}\simeq {\bf 26}\oplus{\bf 1}$.
Explicitly, a vector in ${\bf 26}$ is described as
\bea \Psi=\begin{bmatrix}
           u+\frac{s}{\sqrt2}e_{10} \\ \psi \\ is
          \end{bmatrix}\label{Psi_26}\eea
where $u\in\BB{R}^9$, while the $10^{th}$ and $27^{th}$ components are correlated so that $\Psi\perp\Psi_{\emptyset}$.

Similarly, if we take ${\bf 27}^*$ of $\FR{e}_6$, we have the decomposition ${\bf 27}^*={\bf 26}^*\oplus{\bf 1}^*$.
Explicitly
\bea
\Phi=
\begin{bmatrix}
  u+\frac{s}{\sqrt2}e_{10} \\ \phi \\ -is \\
\end{bmatrix},~~~\Phi_{\emptyset}=\frac{\sqrt2}{\sqrt3}\begin{bmatrix}
             -ie_{10} \\ 0 \\ \frac{1}{\sqrt2}
            \end{bmatrix}.\nn\eea
Even though ${\bf 27}\not\simeq{\bf 27}^*$ for $\FR{e}_6$, we have
\begin{lemma}\label{lem_26_real}
The ${\bf 26}$ and ${\bf 26}^*$ representations of $F_4$ are isomorphic, with the explicit isomorphism
\bea {\bf 26}^*\ni \Phi\stackrel{\gs}{\to} \gs\Phi=\begin{bmatrix}
  u-\frac{1}{\sqrt2}se^{10} \\ ie^{10}\phi \\ -is
\end{bmatrix}\in{\bf 26}.\nn\eea
The same $\gs$ also gives ${\bf 1}\simeq {\bf 1}^*$.
\end{lemma}
\begin{proof}
The proof consists of showing that the map above is a map of representations. It is quite clear that the map is compatible with the action of $\FR{so}(9)$. So we need only check the action of $\xi+\eta$. We assume that $C\xi^*=\eta$ so that the Lie algebra generator $\xi+\eta$ be real. As a consequence $e^{10}C\xi^*=-i\xi$ and $e^{10}C\eta^*=i\eta$, and we check the action
\bea (\xi+\eta)\circ\Phi&=&\begin{bmatrix}
                          \frac{1}{\sqrt2}(\eta e^i\phi)e^i+\frac{1}{\sqrt2}(\eta e^{10}\phi)e^{10} \\
                          -\frac{3is}{2}\eta-\frac{1}{\sqrt2}u\xi \\
                          -(\xi\phi)
                         \end{bmatrix}=\begin{bmatrix}
                          \frac{i}{\sqrt2}(\xi e^ie^{10}\phi)e^i+\frac{i}{\sqrt2}(\eta e^{10}\phi)e^{10} \\
                          -\frac{3is}{2}\eta-\frac{1}{\sqrt2}u\xi \\
                          -i(\eta e^{10}\phi)
                         \end{bmatrix}\nn\\
                         &\stackrel{\gs}{\to}&\begin{bmatrix}
                          \frac{i}{\sqrt2}(\xi e^ie^{10}\phi)e^i-\frac{i}{\sqrt2}(\eta e^{10}\phi)e^{10} \\
                          -\frac{3is}{2}\xi-\frac{1}{\sqrt2}u\eta \\
                          -i(\eta e^{10}\phi)
                         \end{bmatrix}
=(\xi+\eta)\circ \begin{bmatrix}
           u -\frac{s}{\sqrt2}e_{10} \\ ie^{10}\phi \\ -is \end{bmatrix}=(\xi+\eta)\circ(\gs\Phi),\nn\eea
           where the $i$ index only goes from 1 to 9.
\end{proof}

\section{Jordan Algebra and the Cubic Invariant}
\subsection{The $\diamond$ Product}
We saw earlier that the dim 27 representation of $E_6$, there is a totally symmetric cubic form $d(\textrm{-},\textrm{-},\textrm{-})$, which can be regarded as a map ${\bf 27}\otimes^s{\bf 27}\to {\bf 27}^*$.
But as representation of $F_4$, they are isomorphic via the map $\gs$ in lem.\ref{lem_26_real}. As a result
\begin{definition}\label{def_diamond_map}
  We have a symmetric $F_4$-invariant pairing $\bra\Psi_1,\Psi_2\ket$, defined by sending $\Psi_2$ to $\gs^{-1}\Psi_2\in{\bf 27}^*$, where $\gs$ is the map in lem.\ref{lem_26_real}, followed by a natural pairing with $\Psi_1$.
  The concrete formula is
  \bea \bra[u_1+\frac{i}{\sqrt2}t_1e^{10};\xi_1;\eta_1;s_1],[u_2+\frac{i}{\sqrt2}t_2e^{10};\xi_2;\eta_2;s_2]\ket
  =u_1\cdotp u_2+s_1s_2+\frac12t_1t_2+(\xi_1\xi_2)+(\eta_1\eta_2).\nn\eea
  We have a bi-linear map of representations of $F_4$
  \bea {\bf 27}\otimes^s{\bf 27}\to {\bf 27},\nn\eea
  denoted as $\Psi_1\diamond\Psi_2$:
  \bea \Psi_1\diamond\Psi_2=\gs d(\Psi_1,\Psi_2,\textrm{-}).\nn\eea
The concrete formula for $\diamond$ is given in \eqref{diamond_product}.
\end{definition}
We will relate next the diamond product to a certain exceptional Jordan product.

\subsection{Exceptional Jordan algebra}\label{sec_EJA}
Consider $3\times 3$ Hermitian Octonion matrices
\bea J=
\begin{bmatrix}
  a & q_1 & q_3 \\
  \bar q_1 & b & q_2 \\
  \bar q_3 & \bar q_2 & c \\
\end{bmatrix}\nn\eea
where the bar denotes the operation in def.\ref{def_bar}. For real octonions, this is the conjugation, but for $\BB{C}\otimes\BB{O}$, this conjugation does not touch the $\BB{C}$ factor.
\begin{definition}\label{def_Jordan_prod}
The Jordan product is defined as the anti-commutator
\bea J_1\star J_2:=\frac12\{J_1,J_2\}.\label{jordan_prod}\eea
This is the unique exceptional Euclidean Jordan algebra, known as the Albert algebra, according to the classification of Jordan, von Neumann and Wigner \cite{JordanNeumannWigner}.
\end{definition}

We now prove thm.\ref{thm_star_diamond}, which shows that the Jordan product above is the diamond product, induced from the invariant tensor $d$.
\begin{proof}(Of thm.\ref{thm_star_diamond})
We recall the notation, a vector $\Psi\in{\bf 26}\oplus{\bf 1}$ is denoted as
\bea\Psi=
  \begin{bmatrix}
    u\oplus \frac{i}{\sqrt2}te^{10} \\ \psi \\ s
  \end{bmatrix},\nn\eea
where $u\in\BB{R}^9$ and $\Psi\in \Gd^+_{10}$.
We say $\Psi$ is real if $u,t,s$ are real and $\psi=ie^{10}C_{10}\psi^*$. There is an inner-product
\bea \bra\Psi_1,\Psi_2\ket&=&u_1\cdotp u_2+\frac12t_1t_2-i(\psi_1 e^{10}\psi_2)+s_1s_2\nn\\
&=&{\tt u}_1\cdotp{\tt u}_2+\frac12(r_1r_2+t_1t_2)+(\eta_1\eta_2)+(\xi_1\xi_2)+s_1s_2,\label{inner_prd_26}\eea
where we have further split $u$, and $\psi$ in terms of 8 dimension quantities
\bea u={\tt u}\oplus \frac{r}{\sqrt2}e^9;~~~~~~
\psi=
\begin{bmatrix}
  \xi \\ \eta
\end{bmatrix},~~~\xi\in\Gd_8^+,~~\eta\in\Gd^-_8\nn\eea
When $\Psi$ is real, then $\xi,\eta$ are Majorana, i.e. $\xi=C_8\xi^*$ and $\eta=C_8\eta^*$.

We arrange $\Psi$ into a matrix as in \eqref{arrangement}
\bea J(\Psi)=\begin{bmatrix}
  \frac{1}{2}(r-s) & \frac{1}{\sqrt2}\xi & \frac{1}{\sqrt2}{\tt u} \\
  \frac{1}{\sqrt2}\bar\xi & \frac12(s-t) & \frac{1}{\sqrt2}\eta \\
  \frac{1}{\sqrt2}\bar {\tt u} & \frac{1}{\sqrt2}\bar\eta & -\frac{1}{2}(r+s) \\
\end{bmatrix}.\nn\eea
\begin{remark}
We remind the reader that for $\xi,\eta$, we implicitly use the maps \eqref{map_s}, \eqref{map_t} to identify them with octonions. For example $\xi,\eta$ should really read $t^{-1}(\xi),\,s^{-1}(\eta)\in\BB{R}^8$. We do not make these maps explicit, due to our unfortunate choice of notation.

The bars are placed in the matrix precisely so that when multiplying two matrices, the resulting octonion multiplication is naturally realised as operations in the Clifford algebra. For example, from prop.\ref{prop_oct_mult_acatar}, we have $J_{21}J_{13}=t^{-1}(\bar\xi)\star {\tt u}=s^{-1}({\tt u}\xi)$, i.e. $\star$ is simply the Clifford action. Similarly $J_{12}J_{21}=t^{-1}(\xi)\star t^{-1}(\bar\xi)=(\xi\xi)$, i.e. the diagonal entries of the product will be given by the natural pairings of spinors.
\end{remark}

First, since $\Psi$ can be projected to the ${\bf 1}$ summand
\bea \Psi \to \bra\Psi,\Psi_{\emptyset}\ket=\frac{s+t}{\sqrt3},\nn\eea
while the naive trace of $J(\Psi)$ gives $\Tr[J(\Psi)]=-(s+t)/2$, so we conclude that the naive trace
\bea \Tr[J(\Psi)]=-\frac{\sqrt3}{2}\bra\Psi,\Psi_{\emptyset}\ket\label{trace}\eea
and so is $F_4$ invariant.

Next we want to show that
  \bea J_1\star J_2=-\frac12J(\Psi_1\diamond\Psi_2)+\frac14\bra \Psi_1,\Psi_2\ket {\bf 1}_3.\label{master_eq_apx}\eea
  It is actually a direct computation. For example for 13 entry
  \bea 2(J_1\star J_2)_{13}&=&-\frac{1}{\sqrt2}s_1{\tt u}_2+\frac12\xi_1\star \eta_2+(1\leftrightarrow2)
  =-\frac{1}{\sqrt2}s_1{\tt u}_2+\frac12(\xi_1{\tt e}^i\eta_2){\tt e}^i+(1\leftrightarrow2)\nn\\
  \sqrt2J(\Psi_1\diamond\Psi_2)_{13}&=&({\tt u}_1s_2+{\tt u}_2s_1)-\frac1{\sqrt2}(\psi_1{\tt e}^i\psi_2){\tt e}^i
  =({\tt u}_1s_2+{\tt u}_2s_1)-\frac1{\sqrt2}(\xi_1{\tt e}^i\eta_2){\tt e}^i-\frac1{\sqrt2}(\xi_2{\tt e}^i\eta_1){\tt e}^i.\nn\eea
  We see that the 13 entry of \eqref{master_eq_apx} match.
  Note that to evaluate $(\psi_1e^i\psi_2)$ based on our explicit choice of Clifford algebra representation, we have
  \bea (\psi_1e^i\psi_2)=\left\{\begin{array}{cc}
                                    (\xi_1 {\tt e}^i\eta_2)+(\xi_2 {\tt e}^i\eta_1) & i=1,\cdots,8 \\
                                    -(\eta_1 \eta_2)+(\xi_1 \xi_2) & i=9 \\
                                    i(\eta_1 \eta_2)+i(\xi_1 \xi_2) & i=10 \\
                                  \end{array}\right.\nn\eea
  Checking the 12 entry (here $J_{1,2}:=J(\Psi_{1,2})$)
  \bea 2(J_1\star J_2)_{12}&=&\frac{1}{2\sqrt2}(r_1-t_1)\xi_2+\frac12{\tt u}_1\star \bar\eta_2+(1\leftrightarrow2)
  =\frac{1}{2\sqrt2}(r_1-t_1)\xi_2+\frac12{\tt u}_1\eta_2+(1\leftrightarrow2)\nn\\
  \sqrt2J(\Psi_1\diamond\Psi_2)_{12}&=&-\frac{1}{\sqrt2}{\tt u}_1\eta_2+\frac{1}{2}(t_1-r_1)\xi_2+(1\leftrightarrow2)  .\nn\eea
  It matches again. Let us check also the 23 entry
  \bea 2(J_1\star J_2)_{23}&=&-\frac{1}{2\sqrt2}(r_1+t_1)\eta_2+\frac12\bar\xi_1\star{\tt u}_2+(1\leftrightarrow2)
  =-\frac{1}{2\sqrt2}(r_1+t_1)\eta_2+\frac12{\tt u}_1\xi_2+(1\leftrightarrow2)\nn\\
  \sqrt2J(\Psi_1\diamond\Psi_2)_{23}&=&-\frac{1}{\sqrt2}{\tt u}_1\xi_2+\frac{1}{2}(t_1+r_1)\eta_2+(1\leftrightarrow2).\nn\eea
  Match again.

  The diagonal entries are more complicated, we compute the lhs below
  \bea 4(J_1\star J_2)_{11}&=&\frac12(r_1-s_1)(r_2-s_2)+\xi_1\star\bar\xi_2+{\tt u}_1\star\bar{\tt u}_2+(1\leftrightarrow2)\nn\\
  &=&(r_1-s_1)(r_2-s_2)+2(\xi_1\xi_2)+2{\tt u}_1\cdotp{\tt u}_2,\nn\\
  4(J_1\star J_2)_{22}&=&\frac12(s_1-t_1)(s_2-t_2)+\bar\xi_1\star\xi_2+\eta_1\star\bar\eta_2+(1\leftrightarrow2)\nn\\
  &=&(s_1-t_1)(s_2-t_2)+2(\eta_1\eta_2)+2(\xi_1\xi_2),\nn\\
  4(J_1\star J_2)_{33}&=&\frac12(r_1+s_1)(r_2+s_2)+\bar\eta_1\star\eta_2+\bar {\tt u}_1\star {\tt u}_2+(1\leftrightarrow2)\nn\\
  &=&(r_1+s_1)(r_2+s_2)+2(\eta_1\eta_2)+2{\tt u_1}\cdotp{\tt u_2}.\nn\eea
  On the other hand
  \bea
  2J(\Psi_1\diamond\Psi_2)_{11}&=&r_{(1}s_{2)}-(\psi_1 e^9\psi_2)-{\tt u}_1\cdotp{\tt u}_2-\frac12r_1r_2+\frac12 t_1t_2\nn\\
  &=&r_{(1}s_{2)}-(\xi_1\xi_2)+(\eta_1\eta_2)-{\tt u}_1\cdotp{\tt u}_2-\frac12r_1r_2+\frac12 t_1t_2,\nn\\
  2J(\Psi_1\diamond\Psi_2)_{22}&=&{\tt u}_1\cdotp{\tt u}_2+\frac12(r_1r_2-t_1t_2)+t_{(1}s_{2)}-(\xi_1\xi_2)-(\eta_1\eta_2),\nn\\
  2J(\Psi_1\diamond\Psi_2)_{33}&=&-{\tt u}_1\cdotp{\tt u}_2-\frac12(r_1r_2-t_1t_2)-r_{(1}s_{2)}+(\xi_1\xi_2)-(\eta_1\eta_2).\nn\eea
  Comparing two sides, we get
  \bea 4(J_1\star J_2)+2J(\Psi_1\diamond \Psi_2)=\frac12(r_1r_2+t_1t_2)+s_1s_2+(\xi_1\xi_2)+(\eta_1\eta_2)+{\tt u}_1\cdotp{\tt u}_2.\nn\eea
  Comparing with \eqref{inner_prd_26}, we get \eqref{master_eq_apx}.
\end{proof}

%\bibliographystyle{alpha}
%\bibliography{../../../../Bib_file/Bib_common}

\newcommand{\etalchar}[1]{$^{#1}$}

\end{document}